%% file: main.tex
\documentclass[12pt, a4paper]{article}

\input{setup}

\begin{document}

\maketitle

\begin{abstract}
\setstretch{1.05}
Estimating the Hüsler-Reiss precision matrix is a fundamental problem for statistical inference in multivariate extremes. In high-dimensional settings, the number of unknown parameters grows quadratically with the dimension, making regularisation indispensable. Existing approaches regularise the estimation problem by exploiting sparsity. 
In this paper, we consider an alternative structural assumption, namely that the Hüsler-Reiss precision matrix is block-structured. To estimate such a matrix, we introduce a new regularisation framework based on a convex fusion penalty. By encouraging rows and columns to merge, this approach provides a parsimonious representation of the precision matrix, allowing for the simultaneous estimation of its coefficients and the underlying partition of the variables.
The resulting convex optimisation problem is solved by an efficient algorithm combining gradient-based updates with progressive fusion steps. We establish non-asymptotic concentration bounds for the empirical weights entering the penalty and prove consistency of both block recovery and precision matrix estimation under suitable regularity conditions. Numerical experiments demonstrate that our methodology accurately recovers the latent block structure while accurately estimating the Hüsler-Reiss precision matrix across various configurations, illustrating the practical benefits of fusion-based regularisation for multivariate extremes. These benefits are also demonstrated by applying the proposed method to foreign exchange data.
\end{abstract}


\input{sections/introduction.tex}
\input{sections/background.tex}
\input{sections/method.tex}
\input{sections/simulation.tex}
\input{sections/conclusion.tex}

{
\setstretch{1.02}
\bibliographystyle{apalike}
\bibliography{references}
}
\input{sections/appendix.tex}

\end{document}

%% file: setup.tex
\usepackage{templates/macros}
\usepackage[a4paper,top=4cm,bottom=4cm,left=2.5cm,right=2.5cm]{geometry}
\usepackage[utf8]{inputenc}
\usepackage[T1]{fontenc}
\usepackage{comment}
\usepackage{dsfont}
\usepackage{graphicx}
\usepackage{hyperref}
\usepackage{amsmath,amsthm,amsfonts,amssymb,amscd}
\usepackage{booktabs}   
\usepackage{array}
\usepackage[final,nopatch=footnote]{microtype}
\usepackage{tikz}
\usetikzlibrary{matrix,decorations.pathreplacing,calc}
\usepackage[skip=2pt]{caption}
\usepackage{subcaption}
\usepackage{algorithm}
\usepackage{algorithmic}
\usepackage{enumitem}
\usepackage{setspace}
\usepackage{authblk}
\usepackage{natbib} 
\date{}
\usepackage{lipsum}

\title{Inference for Hüsler-Reiss block models}

\author[1,2]{CAPEL Alexandre\footnote{Address for correspondance: CAPEL Alexandre, IMAG, Montpellier, France.  Mail: alexandre.capel@umontpellier.fr}}
\author[1]{DEMANGEOT Marine}
\author[1,2]{MEYER Nicolas}
\author[1,2]{TOULEMONDE Gwladys}

\affil[1]{IMAG, Université de Montpellier, France}
\affil[2]{LEMON, Inria, Montpellier, France}

\usepackage{ulem} 
\usepackage{cancel} 

\usepackage{tikz}

\usetikzlibrary{calc}

%% file: sections/introduction.tex
\section{Introduction}

Extreme value theory (EVT) provides a probabilistic framework for modelling rare events and quantifying the risk associated with large observations. In a multivariate setting, the dependence structure between extremes plays a key role, as the simultaneous occurrence of large values often drives the overall risk. A common approach consists in studying threshold exceedances, whose limit is, under mild regularity conditions, a multivariate generalised Pareto distribution \citep{rootzenMultivariateGeneralizedPareto2006}. This limit is the counterpart to the max-stable distributions, which appear as the limit of the normalised maxima (see \cite{de1977limit}, or the monographs of \cite{de_haan_fereira_2006} and \cite{beirlant_et_al_2006}).
It provides a flexible non-parametric model for multivariate extremes. As the dimension increases, however, estimating the dependence structure becomes increasingly challenging, motivating the development of parsimonious models and dimension reduction techniques for multivariate extremes (\cite{engelkeSparseStructuresMultivariate2021}).

Among the available parametric models for multivariate extremes, the Hüsler-Reiss distribution occupies a particularly prominent position. Originally introduced as the limiting distribution of normalised maxima of Gaussian triangular arrays by \cite{huslerMaximaNormalRandom1989}, it has become one of the most widely used models for threshold exceedances because of both its flexibility and its close connection with Gaussian dependence structures. This distribution is parameterised by a variogram matrix which completely characterises its dependence structure. Beyond its probabilistic interpretation, this parameterisation has profoundly changed statistical inference for Hüsler-Reiss models. An equivalent and often more convenient representation is provided by the Hüsler-Reiss precision matrix $\Theta$, introduced by \cite{engelkeGraphicalModelsExtremes2020}. Similarly to the Gaussian case, this matrix provides a compact description of the dependence structure and constitutes the main object of inference in the Hüsler-Reiss model.

The introduction of the Hüsler-Reiss precision matrix has led to the development of graphical models for extremes. In Gaussian graphical models, the precision matrix encodes conditional independence between variables \citep{lauritzenGraphicalModels1996}. Extending this idea to extremes is non-trivial since for max-stable distributions conditional independence implies independence \citep{papastathopoulos2016conditional}. This limitation was overcome by \cite{engelkeGraphicalModelsExtremes2020} who introduced graphical models for multivariate generalised Pareto distributions together with an appropriate notion of extremal conditional independence. In the Hüsler-Reiss setting, this notion is captured by the precision matrix: a null entry $\Theta_{ij}$ indicates that variables $i$ and $j$ are conditionally independent given the remaining variables, see \textit{e.g.} \cite{engelkeGraphicalModelsExtremes2020} and \cite{hentschelStatisticalInferenceHusler2025}.

Estimating the Hüsler-Reiss precision matrix is particularly challenging in high-dimensional settings where the number of parameters grows quadratically with the dimension.
Most existing methods rely on sparsity as a structural assumption, that is, zero entries in the precision matrix, thereby recovering the underlying extremal graphical model, see \cite{engelke2022structure} and \cite{wan_zhou}.
In contrast, \cite{lederer_oesting} introduce a score-matching-based framework that bypasses intractable normalising constants, yielding scalable estimators for the precision matrix, with theoretical guarantees in high dimensions.
Other approaches have also been proposed for estimating the Hüsler-Reiss model under some specific constraints such as total positivity \cite{rottger_engelke_zwiernik} or a local metric property \cite{rottger_schmitz}.

While sparsity provides an effective form of regularisation, it may fail to exploit richer structural patterns arising in high-dimensional settings. In many applications, variables may naturally organise into clusters that exhibit similar extremal dependence patterns. In this work, we formalise this idea by assuming a block structure for the Hüsler-Reiss precision matrix. The relevant structure is therefore not characterised by isolated zero coefficients, but by groups of variables sharing common extremal dependence. Such block structures provide a parsimonious representation of the dependence by substantially reducing the number of parameters while preserving the main features of the extremal dependence. Moreover, they yield an interpretable clustering of the variables, offering additional insight into the organisation of the underlying tail structure. Although block-structured covariance and precision matrices have been extensively investigated in classical multivariate statistics, similar regularisation strategies have received comparatively little attention for extremal multivariate models.
This block structure assumption naturally motivates a different estimation problem. Specifically, instead of identifying zero entries of the Hüsler-Reiss precision matrix, one seeks to recover groups of rows and columns that are identical up to the symmetry constraints imposed by the matrix. The purpose of this paper is thus to develop a statistical methodology for estimating Hüsler-Reiss precision matrices, addressing not only parameter estimation itself but also the recovery of their unknown block structure. We replace the usual sparsity-inducing penalty by a regularisation strategy that promotes equality within groups of the coefficients, thereby inducing the desired block structure of the precision matrix. Our approach relies on a Clusterpath-type penalty originally developed by \cite{hockingClusterpathAlgorithmClustering2011} for convex clustering  and applied to Gaussian graphical models by \cite{touw2026clusterpath}. The proposed regularisation produces a sequence of estimators with increasingly coarse block structures as the regularisation parameter grows, thereby simultaneously estimating the precision matrix and the underlying partition of the variables.

Following the ideas introduced in the Gaussian case, we develop an efficient optimisation algorithm combining gradient-based updates with progressive fusion of rows and columns along the regularisation path. The algorithm automatically identifies the emerging block structure while preserving the positive semi-definite constraints required by the model.

Beyond the methodological contribution, an important objective of this work is to establish statistical guarantees for the proposed procedure. Building upon recent concentration results for empirical variogram estimators for Hüsler-Reiss models (see \cite{engelke2022structure} and \cite{engelke2026learning}), we derive non-asymptotic bounds for the empirical weights which appear in the penalty and prove consistency of the estimated block structure under suitable regularity conditions. Our results identify conditions under which the procedure achieves consistent partition recovery and consistent estimation, providing theoretical insight into the trade-off induced by the fusion regularisation.

The contributions of this paper are fourfold. We introduce a new regularisation framework for estimating block-structured Hüsler-Reiss precision matrices, thereby extending the scope of structured estimation beyond sparsity. Second, we develop a dedicated optimisation algorithm together with a refitting strategy that improves estimation accuracy after cluster recovery. Third, we establish theoretical guarantees for both the empirical weights and the recovery of the underlying block structure. Finally, through an extensive simulation study, we demonstrate that the proposed procedure accurately estimates the precision matrix while reliably identifying the latent partition under a variety of balanced, unbalanced and perturbed block configurations.

Section \ref{background} introduces Hüsler-Reiss models and the class of block-structured matrices considered throughout the paper. Section \ref{method} presents the proposed regularised estimator together with the associated optimisation algorithm. We establish in Section \ref{framework} the asymptotic properties of the estimator, which follows from the finite sample estimation error, including consistency results for both the recovery of the block structure and the estimation of the Hüsler-Reiss precision matrix. Section \ref{sect-sim} reports numerical experiments illustrating its finite-sample performance. Finally, the use of this method is illustrated by a practical example in Section \ref{sect-app}. All proofs and additional details are deferred to the Supplementary Materials.

%% file: sections/background.tex
\section{Theoretical background}
\label{background}
\subsection{Notation}

Throughout this paper, the following notations are used. The set of real $d \times d$ matrices is denoted by $\mathcal{M}_d(\mathbb{R})$ and, for any $M \in \mathcal{M}_d(\mathbb{R})$, its transpose is denoted by $M^\top$. Bold symbols such as $\vect{x} \in \R^d$ are vectors where each component is written as $x_i,~ i \in V = \{1, \dots, d\}$. For example, $\vect{0}_d$ or $\vect{1}_d$ are respectively the $d$-dimensional vectors $(0, \dots, 0)^\top$ and $(1, \dots, 1)^\top$. All the operations over vectors are meant componentwise, and $\vect{a} < \vect{b}$ means that $a_i < b_i$ for all $i \in V$.  Finally, for any $p \in [1, \infty]$, $||\cdot||_p$ denotes the $l_p$-norm and $|| \cdot ||_F$ stands for the Frobenius norm.

\subsection{Hüsler-Reiss distribution}

Multivariate extreme value theory offers two complementary asymptotic frameworks for describing the extremal dependence of a random vector $\vect{X} \in \mathbb{R}^d$. The first one studies suitably normalised componentwise maxima of independent replicates of $\vect{X}$, which converges to a max-stable distribution under mild assumptions, see \textit{e.g.} \citet{de1977limit}. The second, adopted throughout this paper, focuses on the exceedances of $\vect{X}$ above a high threshold. Since our primary interest lies in the dependence structure, we assume throughout that the marginal distributions are known and have been standardised to the standard Pareto scale. Under this marginal standardisation and a multivariate regular variation assumption (see \citet{resnick2008extreme}), the conditional distribution of threshold exceedances converges to a multivariate Pareto distribution. More precisely,
\begin{equation}
\label{eq:conv-mgpd}
     \lim_{u \rightarrow \infty} \p\left(\vect{X}/u \le \vect{z} \, | \, ||\vect{X}|| > u\right) = \p(\vect{Y} \le \vect{z}) \,,
\end{equation}
for any $\vect{z} \in \mathcal L = \{\vect{z} \in \R^d_+ : ||\vect{z}|| > 1\}$.
The limiting random vector $\vect{Y}$ follows a multivariate Pareto distribution and $\vect{X}$ is said to be in the domain of attraction of $\vect{Y}$, see \citet{rootzenMultivariateGeneralizedPareto2006}. Throughout this work, the infinity norm is used, so that conditioning occurs whenever at least one component exceeds the threshold. This choice naturally connects the multivariate Pareto framework with its associated max-stable distribution (\citet{de1977limit}, \citet{de_haan_fereira_2006} \citet{resnick2008extreme}).

The Hüsler-Reiss distribution  \citep{huslerMaximaNormalRandom1989} is one of the most widely used parametric models for multivariate threshold exceedances. Since the proposed methodology is formulated in terms of the Hüsler-Reiss precision matrix, we first recall the more classical probabilistic representation of the model and establish the notation used throughout the paper.

A Hüsler-Reiss distribution is characterised by a symmetric variogram matrix $\Gamma^\star = (\Gamma^\star_{ij})_{1 \le i, j \le d}$ with zero diagonal coefficients and non-negative entries. Furthermore, $\Gamma^\star$ is strictly conditionally negative definite, that is, $\vect{a}^\top \Gamma^\star\vect{a} < 0$ for all non-zero vectors $\vect{a} \in \R^d$ satisfying $a_1 + \cdots + a_d = 0$, see \citet{kabluchkoExtremesIndependentGaussian2011}. The corresponding density admits a closed-form expression given by
\begin{equation}\label{eq:HR_density}
    f(\vect{y})
    \propto
    y_k^{-2} \prod_{i\in V\setminus\{k\}} y_i^{-1}\phi(\vect{\tilde y}_{\setminus k} ; \Sigma^\star_{(k)})\, \quad \vect{y} \in \mathcal L\,, \quad k \in V\,,
\end{equation}
where $\phi(\cdot ; \Sigma^\star_{(k)})$ is the density of a centered Gaussian vector with covariance matrix $\Sigma^\star_{(k)} \in \mathcal{M}_{d-1}(\R)$, $\vect{\tilde y}_{\setminus k} = (\log(y_i/y_k) + \Gamma^\star_{ik}/2)_{i \in V\setminus \{k\}}$, and
\begin{equation*}
    (\Sigma^\star_{(k)})_{ij} = \frac{1}{2} (\Gamma^\star_{ik} + \Gamma^\star_{kj}-\Gamma^\star_{ij}), \quad \text{for } i,j \neq k\,.
\end{equation*}
Although the density is expressed on an arbitrary reference component $k$, the resulting distribution does not depend on this choice.
The representation \eqref{eq:HR_density} highlights the central role of the variogram matrix, which completely determines the dependence structure of the Hüsler-Reiss distribution. 
We refer to \citet{engelke2015estimation} for a comprehensive presentation of the Hüsler–Reiss model and its link with the associated max-stable Hüsler-Reiss distribution, introduced primarily by \citet{huslerMaximaNormalRandom1989}.

Hüsler-Reiss models are appealing for both practical and theoretical reasons. From an applied perspective, they provide a flexible model for multivariate extremes and have been successfully used in environmental applications such as river discharges, rainfall and spatial extremes, see for instance \citet{asadi2015river,thibaud2013rainfall,davison2013geostatistics} for some applications. From a theoretical perspective, they hold a distinguished position within multivariate extreme value theory because they play a role analogous to Gaussian distributions in classical probability. This follows from the asymptotic result of \citet{huslerMaximaNormalRandom1989}, which states that the limit of the maximum of properly normalised Gaussian random vectors follows a max-stable Hüsler-Reiss distribution. This analogy extends beyond their asymptotic construction. For instance, just as Gaussian distributions are characterised by their covariance matrix, the dependence structure of Hüsler-Reiss models is fully summarised by the variogram matrix. The latter is linked to the pairwise tail dependence coefficient $\chi_{ij}$, which quantifies the strength of extremal dependence between variables $i$ and $j$, as follows:
\begin{equation}
\label{eq:relation_chi_gamma}
\chi^\star_{ij}
= \lim_{u \to 1} \p(F_i(X_i) > u | F_j(X_j) > u)
= \p(Y_i >1 | Y_j >1)
= 2 - 2 \Phi(\textstyle{\sqrt{\Gamma^\star_{ij}}}/2)\,,
\end{equation}
where $ \Phi$ denotes the cumulative distribution function of a standard Gaussian random variable and $F_i$ is the cumulative distribution function of $X_i$. Equation \eqref{eq:relation_chi_gamma} shows that larger variogram values correspond to weaker extremal dependence and, more importantly, that the dependence structure is entirely summarised by pairwise extremal dependence coefficients. As proved by \citet{lalancette2023pairwise}, the Hüsler-Reiss family is in fact the unique multivariate Pareto model enjoying this remarkable property.

Although the variogram matrix provides a natural description of extremal dependence, it is not the most convenient parameterisation for statistical inference. 
A major development in recent years has been the introduction of an equivalent parameterisation through a precision matrix, which plays a role analogous to the one in Gaussian graphical models. Since the proposed methodology is formulated directly in terms of this matrix, we introduce this representation in the next subsection.

\subsection{The Hüsler-Reiss precision matrix}
Instead of relying on the canonical variogram, statistical inference for Hüsler-Reiss models is advantageously framed around an equivalent parameterisation based on the Hüsler-Reiss (HR) precision matrix, denoted here by $\Theta^\star$, which plays a role analogous to the precision matrix in Gaussian graphical models \citep{engelkeGraphicalModelsExtremes2020,hentschelStatisticalInferenceHusler2025}.

\citet{hentschelStatisticalInferenceHusler2025} showed that every Hüsler-Reiss distribution admits the equivalent density representation
\begin{equation*}
f(\vect{y})
\propto
\prod_{i=1}^d y_i^{-1}\exp\Big\{\vect{\mu}_{\Gamma^\star}^\top \log(\vect{y})- \frac 12  \log(\vect{y})^\top \Theta^\star  \log(\vect{y})\Big\}\,, \quad \vect{y} \in \mathcal L\,,
\end{equation*}
where $\Pi = I_d - \vect{1}_d \vect{1}_d^\top / d$ denotes the orthogonal projection on the space $\{\vect{x} \in \R^d: x_1 + \cdots + x_d = 0\}$, and $\vect{\mu}_{\Gamma^\star} = \Pi (-\Gamma^\star / 2) \vect{1}_d$.
The matrix $\Theta^\star$ belongs to the convex cone
\begin{equation*}
    \SPd = \left\{\Theta \in \mathcal M_d(\R): \Theta^\top = \Theta, \, \Theta \text{ is positive semi-definite}, \, \Theta \vect{1}_d = \vect{0}_d, \,\text{rank}(\Theta) = d-1\right\}\,,
\end{equation*}
is unique, and is related to the variogram via the relation $\Theta^\star = (\Pi(-\Gamma^\star /2)\Pi)^+$, where $(\cdot)^+$ denotes the Moore-Penrose pseudo-inverse. Consequently, the variogram matrix and the HR precision matrix are in one-to-one correspondence.

Even if $\Theta^\star$ itself is non-invertible, one can derive from this matrix a sequence of invertible matrices $(\Theta^\star_{(k)})_{k\in V}$, where $\Theta^\star_{(k)}$ is defined as $\Theta^\star$ with the $k$-th row and column being removed. Each matrix $\Theta^\star_{(k)}$ is strictly positive definite and  $(\Theta^\star_{(k)})^{-1} = \Sigma^\star_{(k)}$, the covariance matrix which appears in \eqref{eq:HR_density}, see \citet[Appendix B]{engelkeGraphicalModelsExtremes2020}.
Hence, each sub-matrix can be seen as the precision matrix of a $(d-1)$-dimensional Gaussian distribution.

The analogy with the Gaussian framework mentioned above extends to $\Theta^\star$ as this matrix plays a role analogous to that of the precision matrix of a multivariate normal distribution.
A null coefficient $\Theta^\star_{ij} = 0$, for $i \neq j$ indicates that the $i$-th and $j$-th marginals of $\vect{Y}$ satisfy a conditional independence property given the other margins. 
Strictly speaking, this conditional independence property is defined for the conditional vectors $\vect{Y}^{(k)} = (\vect{Y} \mid  Y_k > 1)$, whose support is a product space, rather than for $\vect{Y}$ itself. We refer to the seminal paper of \citet{engelkeGraphicalModelsExtremes2020} for more details on graphical models for extremes.
Within the framework of graphical models of \citet{lauritzenGraphicalModels1996}, the HR precision matrix can be interpreted as a signed Laplacian matrix of a graph where the variables $Y_1,\dots, Y_d$ correspond to the nodes and the edges are weighted by the matrix entries \citep{rottger_coons_grosdos2026}. This extends the standard Laplacian definition by allowing positive off-diagonal coefficients. In this context, a null entry $\Theta^\star_{ij} = 0$ corresponds to the absence of an edge between nodes $i$ and $j$. 

Besides its graphical interpretation, the HR precision matrix also admits a characterisation as the unique minimiser of the convex objective function $-\log(|\Theta|_+) - \frac 12 \text{tr}(\Theta \Gamma^\star)$,
see \citet{hentschelStatisticalInferenceHusler2025}. Here $|\cdot|_+$ denotes the pseudo-determinant, that is, the product of the non-null eigenvalues of the considered matrix.
In practice, $\Gamma^\star$ is unknown and $\Theta^\star$ can be assessed by minimising
\begin{equation}
    \label{nllh}
    L(\Theta) := -\log(|\Theta|_+) - \frac 12 \text{tr}(\Theta \hat{\Gamma})\,,
 \end{equation}
 where $\hat{\Gamma}$ is an estimator of $\Gamma^\star$.
As brought out by \citet{hentschelStatisticalInferenceHusler2025}, when $\hat \Gamma$ is taken to be the empirical variogram estimator provided by \citet{engelke2022structure}, the objective function $L$ coincides with a surrogate negative log-likelihood for the Hüsler-Reiss model. This function is the direct analogue of the Gaussian negative log-likelihood expressed in terms of the precision matrix, where $\hat{\Gamma}$ is replaced by the sample covariance matrix, see, e.g., \citet{yuan_lin} or \citet{friedman_hastie_tibshirani}.

Existing estimation procedures of $\Theta^\star$ almost exclusively rely on sparsity as the underlying structural assumption. Under this paradigm, regularisation such as LASSO penalty aims at recovering the support of the HR precision matrix and, consequently, the associated extremal graphical model.
In the present work, we depart from this viewpoint by considering a different form of structural regularisation. Rather than promoting sparsity, we assume that the HR precision matrix exhibits a block structure, reflecting groups of variables sharing similar extremal dependence patterns. The next subsection formalises this class of matrices, which constitutes the central structural assumption of our methodology.

\subsection{Hüsler-Reiss block model}
\label{subsec:block-model}

High-dimensional extremal datasets often exhibit latent groups of variables sharing similar dependence patterns. For example, in environmental applications, nearby monitoring stations are expected to experience comparable extremal events, leading to similar tail dependence structures. Rather than estimating all pairwise interactions independently, it is therefore natural to seek a lower-dimensional representation based on clusters of variables with homogeneous extremal behaviour, described by the tail dependence coefficient $\chi$, defined in Equation \eqref{eq:relation_chi_gamma}. Before specifying precisely what this homogeneity entails, it is useful to introduce the notion of a block-structured matrix.

Recall that a partition $\mathcal{C}$ of $V$ corresponds to a family of non-empty disjoint subsets $C_1, \ldots, C_K$ of $V$ such that $C_1 \cup \cdots \cup C_K = V$. In what follows, for such a partition, we denote by $d_1, \ldots, d_K$ the associated cluster size, so that $d_1 + \cdots + d_K = d$.

\begin{definition}[Block-structured matrix]
    \label{def:block_structure}
Let $\mathcal{C} = \{C_1, \ldots, C_K\}$ be a partition of $V$. A symmetric matrix $M \in \mathcal{M}_d(\R)$ is said to satisfy a block structure with respect to $\mathcal{C}$ if it can be written as
\begin{equation*}
\footnotesize{
M
=
\begin{pmatrix} 
    \alpha_1 I_{d_1} & 0 & \dots & 0 \\
    0 & \alpha_2 I_{d_2} & \dots & 0 \\
    \vdots & \vdots & \ddots & \vdots \\
    0 & 0 & \dots & \alpha_K I_{d_K}
\end{pmatrix}
    + U R U^\top\,,}
\end{equation*}
where $U=(u_{ik})_{i \in V, \, k = 1,\dots,K}$ is a $d \times K$ matrix satisfying $u_{ik} = 1$ if $i \in C_k$ and $0$ otherwise, $R \in \mathcal{M}_K(\R)$ is a symmetric matrix of size $K$ called the \textit{reduced matrix}, and $\alpha_1, \ldots, \alpha_K \in \R$.
\end{definition}

The elements $C_1, \ldots, C_K$ of the partition are called \textit{clusters}. For any pair of indices in the same cluster, the coefficients in the corresponding column (and row, to preserve symmetry) are equal. The partition is equivalently encoded by the matrix $U$ which indicates which cluster each coordinate belongs to. In what follows, unless stated otherwise we omit explicit references to the partition and simply refer to $M$ as a block-structure matrix, or say that it is block-structured.

The block structure is entirely determined by three ingredients: the partition $\mathcal{C}$, the reduced matrix $R$, and the coefficients $\alpha_1, \ldots, \alpha_K$. A priori, there is no constraints on these $K$ coefficients. However, in many cases some extra assumptions on the matrix $M$ entail conditions on these parameters.

In our framework, we formalise the notion of homogeneous extremal behaviour across subsets of variables by assuming that these variables can be partitioned into clusters $C_1, \ldots, C_K$ such that those within a given cluster share identical within- and cross-cluster dependencies. Specifically, all pairs of variables belonging to the same cluster share the same tail dependence coefficient, while pairs of variables belonging to two given distinct clusters also share a common tail dependence coefficient. According to Equation \eqref{eq:relation_chi_gamma}, this assumption is equivalent to assuming that the variogram matrix is block-structured with $\alpha_k = -\Gamma_{ij}$ for $i,j \in C_k$, the common coefficient of $\Gamma^\star$ in the cluster $k$. This corresponds to $2 \Phi^{-1}(1-\chi_k/2)^2$ where $\chi_k$ is the joint dependence parameter in the cluster $k$. 

Assuming the variogram is block-structured implies that the family of matrices $\Sigma^{(k)}$ satisfies the same constraint for any $k \in V$. Since the block-structure property is preserved under matrix inversion, the variogram $\Gamma^\star$ is a block-structure matrix if and only if the HR precision matrix $\Theta^\star$ is a block-structure matrix.


In view of this equivalence, we henceforth formulate the methodology directly in terms of the HR precision matrix $\Theta^\star$. As noted earlier, besides being the parameter of interest for inference, it also admits an interpretation in terms of conditional extremal independence.

Under the block-structure assumption, the HR precision matrix $\Theta^\star \in \SPd$ admits the following explicit representation with respect to a partition $\mathcal{C} = \{C_1, \ldots, C_K\}$:
\begin{equation}
\label{eq:block_structure_Theta}
\footnotesize{
\begin{pmatrix} 
    (a_{1} - r_{11})I_{d_1} & 0 & \dots & 0 \\
    0 & (a_{2} - r_{22})I_{d_2} & \dots & 0 \\
    \vdots & \vdots & \ddots & \vdots \\
    0 & 0 & \dots &(a_{K} - r_{KK})I_{d_K}
\end{pmatrix} + U R U^\top\,.}
\end{equation}
Here we have $a_k = r_{kk}- \sum_{\ell=1}^K d_\ell r_{k\ell}$ ensuring the row null-sum property satisfied by $\Theta$. A Hüsler-Reiss distribution with an HR precision matrix satisfying \eqref{eq:block_structure_Theta} is referred to as a \textit{Hüsler-Reiss block model}. Such a structure corresponds to the $K$-block model introduced by \citet{eisenach2020high} in the context of Gaussian graphical models, where $K$ represents the number of clusters, $\mathcal C$ is the clustering partition of $V$ and $d_k$ is the size of cluster $C_k$. We refer to the Supplementary Material \ref{example_block_structure_Theta} for an illustrative example of a block-structure precision matrix. 

In a block-structured model, dependence boils down to two types: within-cluster and cross-cluster dependencies. For a given cluster $C_k$, the former is summarised by the diagonal term $r_{kk}$, whereas for two distinct clusters $C_k$ and $C_\ell$, the latter is given by the off-diagonal term $r_{k\ell}$. These two types of dependence are summarised in the reduced matrix $R$. However, not every block-structured matrix defines a valid HR precision matrix. Supplementary Materials \ref{appendix:characterisation.hr} provides simple sufficient conditions ensuring that the representation \eqref{eq:block_structure_Theta} belongs to $\SPd$.




%% file: sections/method.tex
\section{Clusterpath procedure for Hüsler-Reiss block models}
\label{method}

This section introduces the proposed estimation procedure for the Hüsler–Reiss block models. Our objective is to estimate the HR precision matrix while simultaneously recovering its underlying partition. We first formulate the corresponding penalised optimisation problem and discuss the three ingredients of the penalty, namely the dissimilarity measure, the adaptive weights and the regularisation parameter.  We then describe the optimisation algorithm used to solve the resulting convex problem.

\subsection{Penalised optimisation problem}
\label{sec:prob}

Our objective is to estimate the HR precision matrix while simultaneously recovering the unknown partition introduced in Subsection \ref{subsec:block-model}. The estimator obtained by minimising the objective function $L$ in \eqref{nllh} fails, in general, to capture the underlying partition, unless the estimator $\hat{\Gamma}$ of $\Gamma^\star$ exhibits a specific block structure.
To recover a specific structure on $\Theta^\star$, a natural strategy is to augment the objective function $L$ with a penalty that encourages column-wise equality within each cluster.
Since this constraint requires equality across entire columns rather than shrinking individual coefficients towards zero, a standard $\ell_1$ regularisation is insufficient.
Therefore, we consider a variant of the convex penalisation method originally introduced by \cite{hockingClusterpathAlgorithmClustering2011} for convex clustering and later adapted to the Gaussian setting by \cite{touw2026clusterpath}.
Specifically, the estimator $\hat{\Theta}$ is defined as the solution to the following penalised optimisation problem: 
\begin{equation}
\label{prob-nllh}
    \min_{\Theta} \{  L_\pen(\Theta , \lambda)\}\, \quad \text{such that } \Theta \in \SPd\,,
\end{equation}
with
\begin{equation*}
    L_\pen(\Theta, \lambda) = L(\Theta) + \lambda\mathcal{P}(\Theta)\,,
\end{equation*}
where $\lambda \ge 0$ is the regularisation parameter, $L$ is the function defined in \eqref{nllh} as $L(\Theta) = -\log(|\Theta|_+) - \frac 12 \text{tr}(\Theta \hat{\Gamma})$. The so-called 
\textit{clusterpath penalty} is given by
\begin{equation}\label{eq:penalty}
  \pen (\Theta) = \sum_{i<j} w_{ij} D^2(\vect{\Theta}_{i\cdot}, \vect{\Theta}_{j\cdot})\,,
\end{equation}
with $D$ a convex dissimilarity measure on $\R^d$ and $w_{ij}$ non-negative weights.
Since $D$ is convex, the Clusterpath penalty is itself convex. Combined with the strict convexity of $L$ over $\SPd$, this implies that the problem \eqref{prob-nllh} admits a unique minimiser. Furthermore, note that $L$ is of class $C^2$ on $\SPd$, and so is $L_\pen$. To obtain some of the results presented in Section \ref{method}, the following assumption on the unpenalised function $L$ is also made.

\begin{assumption}
\label{aspt-3}
    The unpenalised objective function $L$ is strongly convex, i.e. there exists a positive constant $S$ such that for all $\Theta \in \SPd$, 
    \begin{equation*}
    \nabla^2 L(\Theta)(\Theta',\Theta') \ge S, \quad \Theta' \in \SPd, ||\Theta'||_F=1,
  \end{equation*}
  where $\nabla^2 L(\Theta)$ denotes the second order differential of $L$ with respect to $\Theta$.
\end{assumption}

This assumption is standard in optimisation theory and guarantees the convergence of gradient descent algorithms, see \cite{garrigos2023handbook}. In our context, it is necessary from a statistical perspective to control the solution of the optimisation problem \eqref{prob-nllh}, see Theorem \ref{theo:estimator_consistency} in Section \ref{framework}.

The construction of the penalty given in Equation \eqref{eq:penalty} relies on three ingredients: a dissimilarity measure between columns, adaptive weights, and a regularisation parameter. We discuss each component in turn.

\paragraph{The dissimilarity measure}

The fusion penalty should compare two columns of the HR precision matrix in a way that reflects the block structure introduced in Subsection \ref{subsec:block-model}.
Since the diagonal entries and the row sums are completely determined by the remaining coefficients, only the free off-diagonal entries need to be compared.
In particular, the dissimilarity should be designed so that it vanishes whenever two variables belong to the same cluster, so that no penalty is incurred within a cluster.

In this paper, we consider
$
D(\vect{\Theta}_{i\cdot}, \vect{\Theta}_{j\cdot}) = \sqrt{\sum_{k\neq i,j} (\Theta_{ik} - \Theta_{jk})^2}
$,
which is a simpler version of the distance used by \cite{touw2026clusterpath}. Although $D$ is not a metric, it satisfies the key property required  by our methodology: two columns have zero dissimilarity if and only if the associated variables belong to the same cluster. It also satisfies the standard triangle inequality.

\paragraph{The weights}
The weights act as a scaling factor that modulates the strength of the fusion penalty.
Pairs of variables that are likely to belong to the same cluster should receive large weights, whereas unlikely pairs should be penalised only weakly.
Specifically, a small dissimilarity $D(\vect{\Theta}_{i\cdot}, \vect{\Theta}_{j\cdot})$ should correspond to a large weight $w_{ij}$, encouraging the variables $i$ and $j$ to be grouped together. Conversely, for a large dissimilarity $D(\vect{\Theta}_{i\cdot}, \vect{\Theta}_{j\cdot})$ the weight $w_{ij}$ should be small, thereby limiting unnecessary fusion between different clusters.
In convex clustering literature, the weights are commonly derived from the squared distance between the data points. We define the family of exponential weights parameterised by a single hyperparameter $\zeta>0$:
\begin{equation}
\label{eq:weights}
  \hat w_{ij} = \exp\left\{- \zeta D^2\left(\hat{\vect{\Gamma}}_{i\cdot}, \hat{\vect{\Gamma}}_{j\cdot}\right)\right\} \in (0,1],
\end{equation}
This choice follows \cite{hockingClusterpathAlgorithmClustering2011} and  \cite{touw2026clusterpath}, while reducing the tuning of the $O(d^2)$ pairwise weights to a single parameter $\zeta$.

The idea behind this construction becomes clear in the ideal setting when $\hat{\Gamma} = \Gamma^\star$. In this case, the weights defined in \eqref{eq:weights} converge to the binary oracle weights $w_{ij}^\star$ as $\zeta$ tends to infinity, i.e.
\begin{equation*}
  \hat w_{ij} \underset{\zeta \rightarrow \infty}{\longrightarrow} w_{ij}^\star = \begin{cases}
    1 \quad \text{if } (i,j) \in T^\star, \\ 0\quad  \text{otherwise}\,,
  \end{cases}
\end{equation*}
where $T^\star = \{(i,j) \in V^2,\, i < j : D(\vect{\Theta^\star}_{i\cdot}, \vect{\Theta^\star}_{j\cdot}) = 0\}$ denotes the set of all pairs of indices that are in the same cluster. The limit weights $w_{ij}^*$ exactly identify pairs belonging to the same block. This asymptotic property plays a central role in the theoretical analysis developed in Section \ref{framework}.

\paragraph{The regularisation parameter}
As mentioned above, the optimisation problem \eqref{prob-nllh} is strictly convex and thus admits a unique minimiser. For a fixed $\lambda>0$, let $\hat \Theta(\lambda)$ denote this unique solution. In what follows, we provide the continuity of the minimiser $\hat \Theta(\lambda)$ with respect to $\lambda$.

\begin{proposition}
  \label{path-sol}
  If there exists $\tilde \lambda>0$ such that $\pen(\hat \Theta(\tilde \lambda)) = 0$, then the mapping
  \begin{align*}
    [0, \tilde \lambda] &\rightarrow \SPd\\
    \lambda &\mapsto \hat \Theta(\lambda)
\end{align*}    
  is continuous, linking the unpenalised solution $\hat \Theta(0)$ to the penalised one $\hat \Theta(\tilde \lambda)$.
\end{proposition}

The continuity of the regularisation path implies that nearby values of the tuning parameter produce similar estimators. Consequently, a relatively coarse grid of regularisation parameters is sufficient in practice.

Once the dissimilarity measure $D$, the weights $(w_{ij})$, and the regularisation parameter $\lambda$ are specified, the matrix minimising \eqref{prob-nllh} can be evaluated. Although solving \eqref{prob-nllh} provides a regularised estimate of the HR precision matrix, it does not explicitly enforce an exact block structure, even with a large $\lambda$. We therefore complement the optimisation step with an explicit fusion procedure that progressively merges sufficiently similar columns. The resulting algorithm is presented below.

\begin{remark}
There is generally no reason for the condition ensuring the smoothness of the solution path in Proposition \ref{path-sol} to hold. However, the additional fusion step ensures that it is satisfied. By imposing merges, it guarantees the existence of a sufficiently large $\lambda$ for which all variables belong to a single cluster, thereby eliminating all pairwise distances.
\end{remark}

\subsection{Algorithmic details of Hüsler–Reiss Clusterpath procedure}

The Hüsler–Reiss Clusterpath (HRC) procedure alternates between two complementary steps. First, clusters that are sufficiently similar are merged according to the dissimilarity measure introduced above. Second, given the current partition, the reduced matrix is updated by minimising the penalised objective function through block gradient descent. These two steps are repeated until convergence.
The subsection details these steps assuming that $K'$ clusters, denoted by $C_1, \dots, C_{K'}$, have already been identified and that the reduced matrix $R$ has been updated accordingly.

\subsubsection{Cluster fusion step}

At each iteration, the algorithm computes the dissimilarity between every pair of current clusters. The two closest clusters, say $C_k$ and $C_\ell$, are merged whenever their dissimilarity falls below a given fusion threshold $\varepsilon_f > 0$, that is, if
$D(C_k, C_\ell) = D(\vect{\Theta}_{i\cdot}, \vect{\Theta}_{j\cdot}) \le \varepsilon_f$.
Since $\Theta$ is block-structured, this dissimilarity does not depend on the particular choice of $i \in C_k$ and $j \in C_\ell$. It can be related to the entries of $R$ as 
\begin{equation*}
D(C_k, C_\ell)^2  =
\tilde D(\vect{r}_{k \cdot}, \vect{r}_{\ell \cdot})^2 :=
\sum_{q \neq k,\ell} d_m (r_{kq}-r_{\ell q})^2 +(d_k-1) (r_{kk}-r_{\ell k})^2+(d_\ell-1) (r_{\ell\ell}-r_{\ell k})^2.
\end{equation*}
If the two clusters are merged, the associated reduced matrix $R$ is updated according to the following rule. By denoting the fused cluster $C_c = C_k \cup C_\ell$, the new reduced matrix $R^{new}$ indexed by $\{1, \dots, K' , c\}\setminus \{k,\ell \}$ is defined as:
\begin{equation}
  \label{r-approx}
  \footnotesize{
r^{new}_{pm} = \begin{cases} 
\frac{d_k r_{km} + d_\ell r_{\ell m}}{d_k + d_\ell} & \text{if } m \neq p = c,\\
\frac{d_k r_{kp} + d_\ell r_{\ell p}}{d_k + d_\ell} & \text{if } p \neq m = c,\\
r_{k\ell} &\text{if } m = p = c,\\
r_{pm} &\text{otherwise}.\\
\end{cases}}
\end{equation}

The coefficients linking the newly created cluster $C_c$ to another cluster are obtained by averaging the previous coefficients, with weights proportional to the cluster sizes. The within-cluster coefficient is inherited from the coefficient linking the two merged clusters, while all remaining entries are left unchanged.

The following proposition shows that this update preserves the admissibility of the resulting new HR precision matrix $\Theta_\text{new}$.

\begin{proposition}
\label{prop:merging_step}
    Let $\Theta \in \SPd$ be a matrix satisfying the block structure \eqref{eq:block_structure_Theta} with associated $C_1, \dots, C_{K'}$ and reduced matrix $R$, and let $\varepsilon_f >0$. Assume that there exist $k,\ell \in \{1,\dots, K'\}$ such that $\tilde D(\vect{r}_{k \cdot}, \vect{r}_{\ell \cdot})\le \varepsilon_f$. Then, for $\varepsilon_f$ sufficiently small, the matrix $\Theta_{new}$ obtained after the merging step described in Equation \eqref{r-approx} belongs to $\SPd$.
\end{proposition}

Proposition \ref{prop:merging_step} guarantees that the fusion step preserves the properties of HR precision matrices, provided that the fusion threshold is sufficiently small. The fusion step is repeated until every pair of clusters is separated by a distance larger than $\varepsilon_f$, after which the reduced matrix is updated through the optimisation step described below.

\subsubsection{Block gradient step}

Once the current partition has been fixed, the reduced matrix is updated by minimising the penalised objective function while the current clustering is fixed. Since only the reduced matrix has to be optimised, the problem can be solved efficiently through block coordinate optimisation developed by \cite{touw2026clusterpath} in the Gaussian case. This procedure is particularly well-suited to the problem, as it exploits the block structure of the current matrix $\Theta$ by automatically computing the gradient block descent using the columns identifying the $K'$ clusters. 

The optimisation proceeds iteratively over the $K'$ clusters. For a given cluster $C_m$, a Newton-Raphson's gradient step is performed on the corresponding column of the reduced matrix $R$, which simultaneously updates all columns of $\Theta$ associated with the variables in $C_m$.
The resulting gradient direction is then computed as
\begin{equation*}
    \vect{\delta}_m = - \nabla_m^2 L_\pen(\Theta, \lambda)^{-1} \nabla_m L_\pen(\Theta, \lambda),
\end{equation*}
where $\nabla_m^2 L_\pen$ and $\nabla_m L_\pen$ are respectively the $m$-th block Hessian matrix and the $m$-th block gradient of $L_\pen$ for a fixed $\lambda$. Explicit expressions of the block gradient and block Hessian are provided in the Supplementary Materials \ref{derivative-sect}.

The gradient direction is combined with a line-search procedure to determine an admissible step size. Besides improving numerical stability, this guarantees that every iterate remains inside the parameter space $\SPd$, using the characterisation established in Proposition \ref{prop:valid-r}.

\subsubsection{Regularisation path computation}

The HRC procedure is based on computing solutions to the problem \eqref{prob-nllh} across a grid of value for $\lambda$. The purpose is therefore to find solutions in which the number of clusters range from $d$ (typically when $\lambda = 0$ and each variable forms its own cluster) to one cluster (when the penalty is null)). However, it is necessary to rescale the penalty for mitigating the procedure's sensitivity to the data.
To this end, let stress that in the current problem, the regularisation parameter is sensitive to the weights of the penalty term. To reduce this sensitivity, we replace $\lambda$ in \eqref{prob-nllh} with its scaled counterpart, $\lambda' = \lambda \bigl(\sum_{i<j} w_{ij}\bigr)^{-1}$.

Finally, for each grid value of $\lambda$, the HRC procedure starts the optimisation algorithm using the solution obtained at the previous $\lambda$. This approach offers two benefits. Firstly, these warm starts enable the inclusion of clustering information for the new solution. Secondly, it will allow the entire procedure to be represented as a hierarchical clustering algorithm, the use of which is demonstrated in Section \ref{sect-app}.

\subsubsection{Refitted estimator}

The alternation between the cluster fusion and the block gradient steps stops when at least $t_{\max}$ iterations have been performed or the absolute value of the difference between two consecutive likelihoods is smaller than a convergence threshold $\varepsilon_c$.
At that point, the algorithm provides a clustering partition and the corresponding reduced matrix $R$.
Since the estimation of the reduced matrix $R$ relies on the penalised negative log-likelihood \eqref{prob-nllh}, it might be biased.
As proposed by \cite{touw2026clusterpath}, a refitted estimate is computed in order to address this issue.
It corresponds to the optimisation of the problem \eqref{nllh} under the cluster partition obtained after the first two steps of the procedure.

\section{Statistical framework}
\label{framework}

Throughout this section, we consider $n$ independent copies 
of a random vector $\vect{X} \in \R^d$ belonging to the domain of attraction of a Hüsler-Reiss distributed random vector $\vect{Y}$ with HR precision matrix $\Theta^\star$. We assume that $\Theta^\star$ satisfies the block structure introduced in \eqref{eq:block_structure_Theta} with associated clusters $C_1, \ldots, C_K$. The marginal distribution functions of $\vect{X}$ are denoted by $F_1,\ldots,F_d$, with empirical counterparts $\hat F_1,\ldots,\hat F_d$.

The estimation procedure developed in Section~\ref{method} is implemented by taking, as the variogram estimator in equations \eqref{nllh} and \eqref{eq:weights}, the empirical estimator proposed by \cite{engelke2022structure}, and computed from the $n$ copies $\vect{X}$. This estimator is denoted by $\hat \Gamma_n$ to emphasise its dependence on the data.
This naturally leads to the empirical penalised criterion
\begin{equation}
  \label{nlikelihood}
    L_{\mathcal P}^{(n)}(\Theta, \lambda) = -\log(|\Theta|_+) - \frac 12 \text{tr}(\hat \Gamma_n\Theta) + \lambda \sum_{i<j}  \hat w_{ij}^{(n)} D^2(\vect{\Theta}_{i \cdot}, \vect{\Theta}_{j \cdot})\,,
\end{equation}
where
\begin{equation}
\label{eq:empiciral_weights}
    \hat w_{ij}^{(n)} = \exp\left\{- \zeta D^2 ((\hat{ \vect{\Gamma}}_n)_{i\cdot}, (\hat{ \vect{\Gamma}}_n)_{j\cdot})\right\}
\end{equation}
denote the empirical weights.

The purpose of this section is to establish statistical guarantees for the unique minimiser $\hat \Theta_n(\lambda) \in \SPd$ of $ L_{\mathcal P}^{(n)}(\Theta, \lambda)$. This quantity is referred to as the Hüsler-Reiss Clusterpath (HRC) estimator.
The analysis proceeds in three steps.
We first recall concentration results for the empirical variogram estimator, which constitutes the only source of randomness in the optimisation problem. We then derive non-asymptotic bounds for the empirical fusion weights and finally show that these results imply both consistent estimation of the HR precision matrix and consistent recovery of its underlying partition.

\subsection{Concentration properties of the empirical variogram}

The theoretical analysis relies on the empirical variogram estimator introduced by \cite{engelke2022structure}. Since our optimisation procedure only depends on the data through this estimator, its concentration properties constitute the key ingredient of all subsequent results. This estimator is defined as
\begin{equation}
  \label{vario}
  \hat \Gamma_n = \frac 1 d \sum_{m \in V} \hat \Gamma_n^{(m)},
\end{equation}
where $((\hat \Gamma_n^{(m)})_{ij})_{1 \le i,j \le d}$ corresponds to the sample variance of $\log (\frac{n}{n+1} - \hat F_i(X_i)) - \log (\frac{n}{n+1} - \hat F_j(X_j))$ conditionally on the event $\{\hat F_m(X_m) >1 - \frac k n\}$.

\begin{remark}
    The threshold parameter $k$ controls the classical bias–variance trade-off inherent to peaks-over-threshold methods by fixing the number of extreme data used to approximate the distribution of $\vect{Y}$. To mitigate this issue, $k=k_n$ is typically selected such that $k \to \infty$ and $k / n \rightarrow 0$ as $n \rightarrow \infty$ so that the tail sample size grows with $n$, at a slower rate than $n$. This guarantees that the peaks-over-threshold approximation remains asymptotically valid while still providing enough data to keep variance manageable.
\end{remark}

To control the asymptotic behaviour of $\hat \Gamma_n$, the following two assumptions are required.

\begin{assumption}
\label{aspt:bounded-behaviour}
    The marginal distribution functions $F_1, \dots, F_d$ of the vector $\vect{X}$ are continuous and there exist positive constants $\tilde K$ and $\beta'$ such that for all $J \subset V$ with $|J| = 2,3$ and $q \in (0, 1]$,
    \begin{equation*}
       \sup_{\vect{x}\in[0,1]^{|J|}}  \left|\frac 1q \, \p(F_J(\vect{X}_J) > \vect{1}- q \vect{x}_J) - \frac{\p(\vect{Y}_J > 1/\vect{x}_J)}{\p(Y_1 > 1)} \right| \le \tilde K q ^{\beta'}\,,
    \end{equation*}
    where $\vect{x}_J=(x_j, j \in J)$.
\end{assumption}

\begin{assumption}
    \label{aspt:vario}
    The coefficients of the variogram $\Gamma^\star$ are finite and strictly positive.
\end{assumption}

These assumptions are standard in multivariate extreme value theory. Assumption~\ref{aspt:bounded-behaviour} is a classical second-order condition controlling the rate of convergence towards the limiting multivariate Pareto distribution. Assumption~\ref{aspt:vario} excludes degenerate dependence structures by requiring all pairwise variogram coefficients to remain strictly positive and finite. In view of \eqref{eq:relation_chi_gamma}, this is equivalent to assuming that every pairwise tail dependence coefficient belongs to $(0,1)$, avoiding some degenerate cases such as asymptotic independence. Together, these assumptions are sufficient to derive concentration inequalities for the empirical variogram estimator, see \cite{engelke2026learning}.

\begin{proposition} \cite{engelke2026learning}
 \label{prop:conc-vario}
 Let Assumptions \ref{aspt:bounded-behaviour} and \ref{aspt:vario} hold and let $\alpha$ be arbitrary in $(0,1)$. Then, for any $\beta<\beta'$, there exist $C, c$ and $M$ depending only on $\alpha$, $\tilde K$ and $\beta$ such that for any $n^\alpha \le k \le \frac n2$ and $\varepsilon \ge M d^3\exp\left(-\frac{ck}{\log(n)^8}\right)$:

  \begin{equation}\label{eq:concentration_bound_variogram}
    \mathbb P \left(||\hat \Gamma_n - \Gamma^\star||_\infty > \delta_n \right) \le \varepsilon,
  \end{equation}
  with 
  \begin{equation}\label{eq:_delta_n_emp_vario}
      \delta_n : = \delta_n(\varepsilon) = C \left\{\left(\frac kn\right)^{\beta} \log\left(\frac nk \right)^2 + \frac{1 + \sqrt{\frac 1c \log(Md^3 / \varepsilon)}}{\sqrt{k}}\right\}.
  \end{equation}
\end{proposition}

Proposition~\ref{prop:conc-vario} shows that the empirical variogram converges uniformly towards its theoretical counterpart at the rate $\delta_n$. As shown in the next subsection, this concentration inequality extends to the empirical weights introduced in \eqref{eq:empiciral_weights}.

We shall remark that, although we focus on the empirical variogram estimator \eqref{vario}, the subsequent analysis only requires a concentration bound of the form \eqref{eq:concentration_bound_variogram}. Consequently, the asymptotic results established in this paper remain valid for any variogram estimator satisfying such a bound.

\subsection{Concentration of the empirical weights}
\label{weights-sect}

The fusion penalty introduced in Section \ref{method} depends on the empirical weights given in \eqref{eq:empiciral_weights}, which serve as estimators for the oracle weights defined in Section \ref{sec:prob} as $w^\star_{ij} = 1$ if $i$ and $j$ belong to the same cluster, and $0$ otherwise. Hence, controlling the estimation error of these weights is essential for analysing the behaviour of the HRC estimator $\hat \Theta_n(\lambda)$. The concentration inequality established above immediately transfers to the empirical weights.

\begin{proposition}
  \label{prop:w-concentration}
   Consider the sequences $(\delta_n)_n$ as given in \eqref{eq:_delta_n_emp_vario} and $(\zeta_n)_n$, which satisfies $\zeta_n =o(\delta_n^{-2})$ and $\zeta_n \rightarrow \infty$ as $n\to \infty$. Let also $\mu^\star = \min_{(i,j) \notin T^\star} D(\vect{\Gamma}^\star_{i\cdot}, \vect{\Gamma}^\star_{j\cdot})$ denote the minimum distance between all different clusters with respect to $\vect{\Gamma}^\star$. We define the sequence     \begin{equation*}
    a_n = \max\left(1-\exp\left\{-4(d-2)^2\zeta_n\delta_n^2\right\}, \exp\left\{-\zeta_n (\mu^\star-2(d-2)\delta_n )^2\right\}\right)\,,
    \end{equation*}
    which converges to $0$ as $n \to \infty$. On the event $\{||\hat \Gamma_n - \Gamma^\star||_\infty \le \delta_n\}$, if $\mu^\star > 2(d-2)\delta_n$ then
    \begin{equation*}
        \max_{1 \le i,j \le d} |\hat w^{(n)}_{ij} - w^\star_{ij}| \le a_n\,.
    \end{equation*}
Consequently, under the conditions of Proposition \ref{prop:conc-vario}, we immediately obtain the following estimation rate:
    \begin{equation*}
    \max_{1 \le i,j \le d} |\hat w^{(n)}_{ij} - w^\star_{ij}| = O_{\p}(a_n)\quad \quad \text{as } n \to \infty \,.
    \end{equation*}
\end{proposition}

Proposition~\ref{prop:w-concentration} shows that the empirical weights converge uniformly towards their oracle counterparts. This result justifies the use of empirical exponential weights in place of the unavailable oracle weights and constitutes the key ingredient in the analysis of the penalised estimator.

\subsection{Consistency of the HRC estimator and partition recovery}

We now establish the statistical guarantees of the HRC estimator $\hat \Theta_n(\lambda)$. We first derive a finite-sample error bound for this estimated HR precision matrix. We then show that this bound implies consistent recovery of the underlying block partition.

To this end, following the notation introduced in Section \ref{method}, recall that the unknown partition $\mathcal{C}$ can be equivalently represented through the set $T^\star$ of pairs of variables belonging to the same block:
\begin{equation*}
T^\star = \left\{(i,j) \in V^2,\, i < j : D(\vect{\Theta}^\star_{i\cdot}, \vect{\Theta}^\star_{j \cdot}) = 0\right\}\,.
\end{equation*}
For $\varepsilon_f > 0$, we consider an estimator $\hat T_\lambda(\varepsilon_f)$ of this set:
\begin{equation*}
\hat T_\lambda(\varepsilon_f) = \left\{(i,j) \in V^2,\, i < j : D(\hat{\vect{\Theta}}_n(\lambda)_{i\cdot}, \hat{\vect{\Theta}}_n(\lambda)_{j \cdot}) \le \varepsilon_f \right\}\,.
\end{equation*}
The estimator $\hat T_\lambda(\varepsilon_f)$ encodes the partition produced by the fusion step with fusion threshold $\varepsilon_f$: a pair of variables belongs to this set if and only if it belongs to the same estimated cluster. 

The following theorem provides a deterministic error bound on the estimation error of the HRC estimator $\hat \Theta_n(\lambda)$. As a consequence, consistency of $\hat\Theta_n(\lambda)$ follows directly from the concentration properties of the empirical variogram.

\begin{theorem}[Error bound and consistency of the estimator]
  \label{theo:estimator_consistency}
  Consider the sequences $(\delta_n)_n$ and $(a_n)_n$ as given in \eqref{eq:_delta_n_emp_vario} and in Proposition \ref{prop:w-concentration}, respectively. Let also $\rho^\star = \max_{i,j \in V} D(\vect{\Theta}^\star_{i \cdot}, \vect{\Theta}^\star_{j \cdot})$ denotes the maximum distance between all different clusters with respect to $\Theta^\star$. We define the sequence
  \begin{equation*}
      b_n(\lambda) := \frac{\delta_n + \lambda 4(d-2) \rho^\star a_n + \sqrt{\Lambda'_n}}{\frac{S}{d^2} -8\lambda (d-2)^2a_n}\,,
  \end{equation*}
  where $\Lambda'_n = (\delta_n + \lambda 4(d-2) \rho^\star a_n)^2 + \frac{2S}{d^2}(\lambda (\rho^\star)^2 a_n +\frac{||\Theta^\star||_1}{d^2} \delta_n)$.
   Under Assumption \ref{aspt-3}, the HRC estimator $\hat\Theta_n(\lambda)$ satisfies
  \begin{equation*}
      ||\hat\Theta_n(\lambda) - \Theta^\star||_\infty \le b_n(\lambda),
  \end{equation*}
   on the event $\{||\hat \Gamma_n - \Gamma^\star||_\infty \le \delta_n\}$, provided that $\lambda \le \frac{S}{8d^2(d-2)^2} a_n^{-1}$.
  Consequently, under Assumption \ref{aspt-3} and the conditions of Proposition \ref{prop:conc-vario}, we obtain the following estimation rate for all $\lambda >0$
  \begin{equation*}
      ||\hat\Theta_n(\lambda) - \Theta^\star||_\infty = O_{\p}\left(\sqrt{\delta_n + \tilde C \lambda a_n}\right)\,,
      \quad \text{as } n \to \infty\,,
  \end{equation*}
  where $\tilde C$ is a constant depending quadratically on the dimension $d$, as well as on $\rho^\star$ and $||\Theta^\star||_1$.
\end{theorem}

Theorem~\ref{theo:estimator_consistency} provides a finite-sample error bound for the HRC estimator on the high-probability event of Proposition~\ref{prop:conc-vario}. The bound naturally decomposes into two contributions: the estimation error of the empirical variogram and the regularisation error introduced by the fusion penalty.

This result immediately translates into a recovery guarantee for the latent partition.
Specifically, the following theorem shows that blocks in $\Theta^\star$ are asymptotically identified by the fusion procedure.

\begin{theorem}[Partition recovery]
  \label{theo:partition_consistency}
  Let $\varepsilon^\star = \min_{(i,j) \notin T}  D(\vect{\Theta}^\star_{i\cdot}, \vect{\Theta}^\star_{j \cdot})$. Under Assumption \ref{aspt-3}, the dissimilarity between two columns of the matrix $\hat \Theta_n(\lambda)$ satisfies
  \begin{equation*}
      \begin{cases}
        D(\hat{\vect{\Theta}}_n(\lambda)_{i\cdot}, \hat{\vect{\Theta}}_n(\lambda)_{j \cdot}) \le 2(d-2) b_n(\lambda),& \text{if } (i,j) \in T^\star, \\
        D(\hat{\vect{\Theta}}_n(\lambda)_{i\cdot}, \hat{\vect{\Theta}}_n(\lambda)_{j \cdot}) \ge \varepsilon^\star - 2(d-2) b_n(\lambda),& \text{otherwise,}
      \end{cases}
  \end{equation*}
  on the event $\{||\hat \Gamma_n - \Gamma^\star||_\infty \le \delta_n\}$, provided that $\lambda \le \frac{S}{8d^2(d-2)^2} a_n^{-1}$.
  
  Consequently, under Assumption \ref{aspt-3} and the conditions of Proposition \ref{prop:conc-vario}, we have for any $\lambda>0$ and any $\varepsilon_f < \varepsilon^\star$, 
  \begin{equation*}
      \p(\hat T_{\lambda}(\varepsilon_f) = T^\star) \to 1\,, \quad \text{as } n \to \infty\,.
  \end{equation*}
\end{theorem}

\begin{remark}
In the error bound in Theorem \ref{theo:estimator_consistency}, the first term $\delta_n$ reflects the statistical accuracy of the empirical variogram $\hat \Gamma_n$ whereas the second term corresponds to the bias introduced by the fusion penalty and increases with the regularisation parameter $\lambda$. In particular, choosing $\lambda =\lambda_n=O(\delta_n)$ yields
    \begin{equation*}
    ||\hat\Theta_n(\lambda_n) - \Theta^\star||_\infty = O_{\p}(\sqrt{\delta_n}),
    \end{equation*}
    thereby approaching the $O_{\mathbb{P}}(\delta_n)$ rate of the unpenalized estimator while promoting the desired block structure. 
\end{remark}

Taken together, the previous results show that the HRC procedure consistently estimates both the HR precision matrix and its latent block structure. The finite-sample behaviour of the estimator is investigated in the next section through an extensive simulation study.

%% file: sections/simulation.tex
\section{Simulation study}
\label{sect-sim}

This section investigates the finite-sample performance of the HRC procedure. More specifically, it pursues two objectives: assessing the accuracy of the HRC estimator and evaluating the ability of the proposed methodology to recover the underlying block structure. Unless otherwise stated, all results are based on $N=100$ independent Monte Carlo replicates.

\subsection{Simulation settings}

For each replicate, we first construct a block-structured HR precision matrix $\Theta^\star$ by specifying a partition $\mathcal C = C_1, \ldots, C_K$ together with its associated reduced matrix $R$. The corresponding variogram matrix $\Gamma^\star$ is then deduced from the one-to-one correspondence between the two parameterisations. 
Samples $\vect{X}_1, \dots, \vect{X}_n$ are generated from a max-stable distribution following the exact simulation algorithm of \cite{dombry2016exact}, implemented through the \texttt{rmstable} function of the \texttt{graphicalExtremes} package. The simulated observations are subsequently transformed into multivariate Pareto data using the \texttt{data2mpareto} function, from which the empirical variogram estimator $\hat \Gamma_n$ of \cite{engelke2026learning} is computed.

The sample size is taken as $n \in \{5000, 10000\}$. In accordance with the asymptotic framework developed in Section \ref{framework}, inference is based on the largest $k = \lfloor n^{0.7}\rfloor$ observations. The weight parameter is fixed to $\zeta = \log(n)^s$ with $s\in \{1, 2\}$, following the theoretical recommendations of Section \ref{method}. Following \cite{chen2015convex} and \cite{Chi02102015}, the performance of the procedure is also improved by assigning zero weights to non-neighbour pairs using the $h$-nearest neighbours. To avoid the introduction of a new parameter, the parameter $h$ is set to $5$.

The convergence tolerance and maximum number of iterations are set to $\varepsilon_c = 10^{-7}$ and $t_{\max} = 1000$, respectively. The fusion threshold is defined as $\varepsilon_f = \kappa \text{ median}_{i,j} D({\hat{\vect{\Theta}}_n(0)}_{i \cdot}, {\hat{\vect{\Theta}}_n(0)}_{j \cdot})$ to account for the scale of the pairwise column dissimilarities, with $\kappa = \log(n)^{-1}$. This data-adaptive choice automatically scales the fusion threshold to the magnitude of the estimated HR precision matrix while gradually decreasing the fusion threshold as the sample size increases.

For each replicate, the regularisation parameter $\lambda$ and the $s$ parameter are selected by a five-fold cross-validation over an $(s, \lambda)$ grid, where $\lambda$ is increased until all the variables are merged. The selected value of $\lambda$ and $s$ minimise the unpenalised negative log-likelihood
\begin{equation*}
\frac 15 \sum_{k=1}^5 \left\{ -\log(|\hat \Theta_{-k}|_+) - \frac 12 \text{tr}(\hat \Theta_{-k}\hat\Gamma_k)\right\},
\end{equation*}
where $\hat \Theta_{-k}$ denotes the HRC estimator without the $k$-th fold and $\hat \Gamma_k$ is the empirical variogram estimated from this fold.

The quality of the estimated HR precision matrix is assessed through the Frobenius norm $\| \hat\Theta_n(\lambda) - \Theta^\star \|_F$. To evaluate the recovery of the block structure, we consider the estimated number of clusters together with the Adjusted Rand Index (ARI), a standard measure of agreement between two partitions, see \cite{rand_1971} for an introduction on this measure.

Throughout the study, we compare three estimators of the HR precision matrix. The first one corresponds to the minimiser of the unpenalised optimisation problem \eqref{nlikelihood}, that is, with $\lambda=0$, and serves as a benchmark. The two other ones are the  HRC estimator and its refitted version defined at the end of Section \ref{method}. These three estimators are denoted \textit{Initial Theta}, \textit{HRC-estimator} and \textit{HR-refit} in the following graphs and comments.

The method is implemented in a \texttt{R} package called \texttt{HRClusterpath} available on \textit{GitHub}\footnote{\url{https://github.com/alxkpl/HRClusterpath}} and the results and figures can be reproduced with the code at \url{https://github.com/alxkpl/HRC-simu}.

\subsection{Results}

The performance of the HRC procedure is investigated under the different scenarios presented below. The results are illustrated graphically in this subsection, while a complete numerical summary is provided in the Supplementary Materials \ref{appendix:table.results}.

\subsubsection{Dimension \texorpdfstring{$d=15$}{d=15}}

We first consider a moderate-dimensional setting ($d=15$), for which we successively consider increasingly challenging scenarios.

We begin with an ideal balanced block model with $K=3$ clusters of size $15$.
The associated reduced matrix $R$ is given by
\begin{equation}
    \label{eq:r-simu}
    \footnotesize{
    R = \begin{pmatrix}
          -0.5 & -0.25 & 0 \\
          -0.25 & -0.75 & -0.25 \\
          0 & -0.25 & -0.5
        \end{pmatrix}
    }
\end{equation}
We then assess the influence of heterogeneous cluster sizes by considering an unbalanced partition while keeping the same reduced matrix $R$ defined in Equation \eqref{eq:r-simu}.
The corresponding clusters are of size $3$, $5$, and $7$.
The numerical results are given in Figure \ref{fig:bal-res15}.

Finally, we investigate another configuration, noised, developed in the Supplementary Materials \ref{appendix:approximate.sim} E.1.

\begin{figure}[!ht]
    \centering
    \includegraphics[width=15cm]{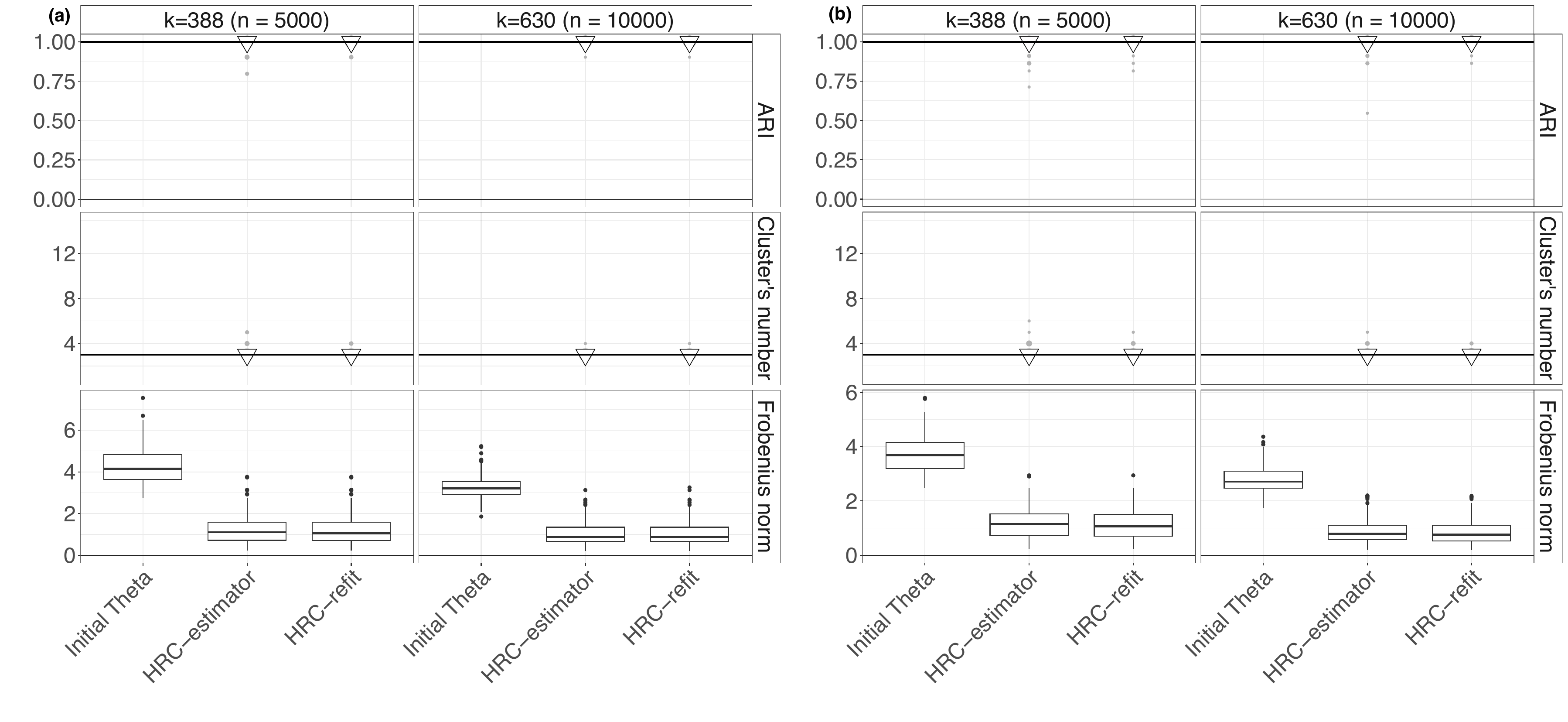}
    \caption{Clustering performance (top two rows) and HR precision matrix estimation error (bottom row) for the balanced (a) and unbalanced (b) structure for $d=15$ and different sample sizes. For the two top rows, the triangle indicates the median over the Monte Carlo replicates. The black lines correspond to ARI perfect clustering (ARI $=1$) and to the true number of clusters, respectively. The bottom row depicts the box-plot of the Frobenius norm $\|\hat \Theta_n(\lambda) - \Theta^\star\|_F$ among the $N=100$ replicates. The clustering metric is omitted for \textit{Initial Theta} because not relevant.}
    \label{fig:bal-res15}
\end{figure}

Several conclusions emerge from Figures \ref{fig:bal-res15}. First, both penalised estimators consistently outperform the unpenalised estimator in terms of Frobenius error, showing that the fusion penalty improves estimation accuracy. Second, block recovery rapidly improves with the sample size and is nearly perfect for $n=10^4$, with ARI values close to one and an almost exact estimation of the number of clusters. Finally, the refitting step systematically improves the HRC estimator by slightly reducing both estimation bias and variability. Similar conclusions hold for balanced and unbalanced partitions, indicating that the HRC procedure is not sensitive to moderate cluster-size heterogeneity.

Taken together, these experiments show that the HRC procedure provides accurate matrix estimation while reliably recovering the underlying block structure. The refitted estimator systematically improves both objectives. It is worth noting that these performances are obtained using only the extreme observations retained by the peaks-over-threshold procedure, namely $k=388$ and $k=630$, corresponding to less than $13\%$ and $7\%$ of the available observations, respectively.

\subsubsection{Dimension \texorpdfstring{$d=60$}{d=60}}

Two scenarios in dimension $d=60$ are proposed. First, we consider a balanced partition with $K=5$ clusters of size $12$, and with reduced matrix
\begin{equation*}
\footnotesize{R
=
\begin{pmatrix}
-0.5 & -0.25 & 0 & 0 & 0 \\
-0.25 & -0.75 & -0.25 & 0 & 0 \\
0 & -0.25 & -0.75 & -0.25 & 0 \\
0 & 0 & -0.25 & -0.75 & -0.25 \\
0 & 0 & 0 & -0.25 & -0.5
\end{pmatrix}\,.}
\end{equation*}
The second scenario corresponds to an unbalanced configuration with $K=3$ clusters of size $10$, $20$, and $30$, respectively. The associated reduced matrix is again the one given in Equation \eqref{eq:r-simu}. The results of these configurations are summarised in Figure \ref{fig:bal-res60}. 

\begin{figure}[!ht]
    \centering
    \includegraphics[width=15cm]{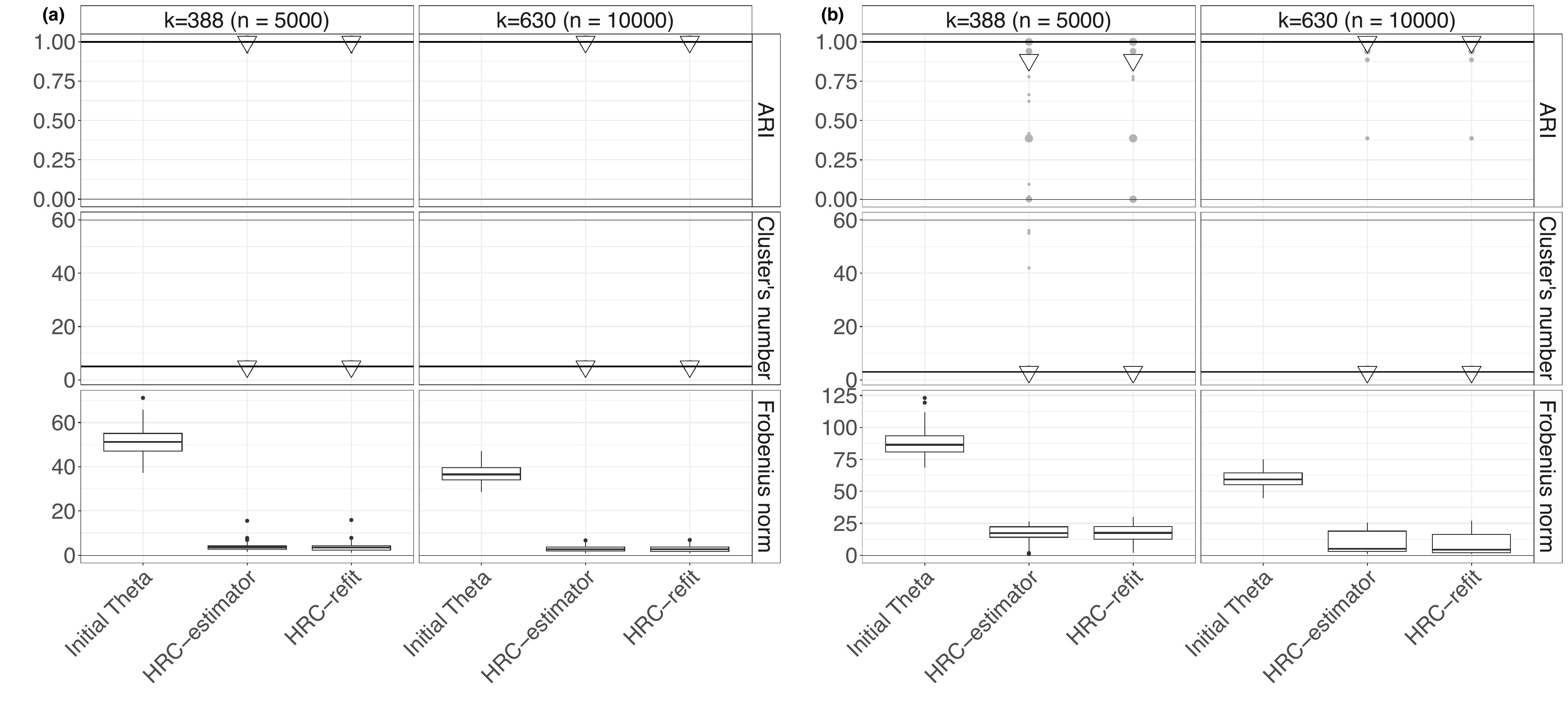}
    \caption{Results with a balanced structure in the coefficients for the balanced with $K=5$ clusters (a) and unbalanced  with $K=3$ clusters (b) for $d=60$ and different sample sizes (column). See Figure \ref{fig:bal-res15} for more details.}
    \label{fig:bal-res60}
\end{figure}

Increasing the dimension from $15$ to $60$ has only a modest impact on the balanced configuration. The clustering accuracy remains almost perfect and the Frobenius error increases only moderately despite the fourfold increase in dimension. The unbalanced configuration is more challenging, especially for $n=5000$, where the variability of the estimated partition increases substantially. Nevertheless, increasing the sample size restores excellent clustering performances.

Overall, the simulation study highlights three main features of our methodology. First, the fusion penalty substantially improves HR precision matrix estimation compared with the unpenalised estimator. Second, the latent block structure is accurately recovered over a broad range of configurations, including heterogeneous cluster sizes and moderate departures from the exact block model. Finally, the refitting step consistently improves both estimation accuracy and clustering performance, making it the preferred estimator in practice.

\section{Application}
\label{sect-app}
We illustrate the proposed HRC procedure on a real-world dataset. The results are presented as a dendrogram, which naturally reflects the hierarchical nature of the algorithm, and are complemented by the analysis of a representative clustering solution which highlights groups of variable with the same dependency structure.

The dataset consists of daily spot foreign exchange rates against the British pound sterling from the $1^{st}$ October 2005 to the $28^{th}$ February 2020, prior to  the COVID-19 pandemic. All the data is available from the website of the Bank of England\footnote{\url{https://www.bankofengland.co.uk/}}.
It contains $n=3642$ daily observations on $d=26$ currencies.
Following \cite{engelke2022structure}, temporal dependence is removed by considering de-GARCHed absolute log-returns, thereby capturing extremes in both positive and negative directions.

The HRC procedure is implemented using the same hyper-parameter settings as in the simulation study. The empirical variogram is computed with $k = \lfloor n^{0.7} \rfloor$, while the adaptive weights are constructed using $h=5$ nearest neighbours and the exponential parameter $\zeta = \log(n)^s$ with $s=1$. The case $s=2$ is not interesting because the induced weights are too low for the HRC procedure to produce meaningful results in terms of hierarchical clustering. A layer of refitted HRC procedure is then added, in order to obtain more reliable results for the coefficients.

To illustrate the possibilities of the procedure, all the clusters, computed along the regularisation parameters, are shown as a dendrogram in Figure~\ref{fig:currency}.

\begin{figure}[!ht]
    \centering
    \includegraphics[width=12cm]{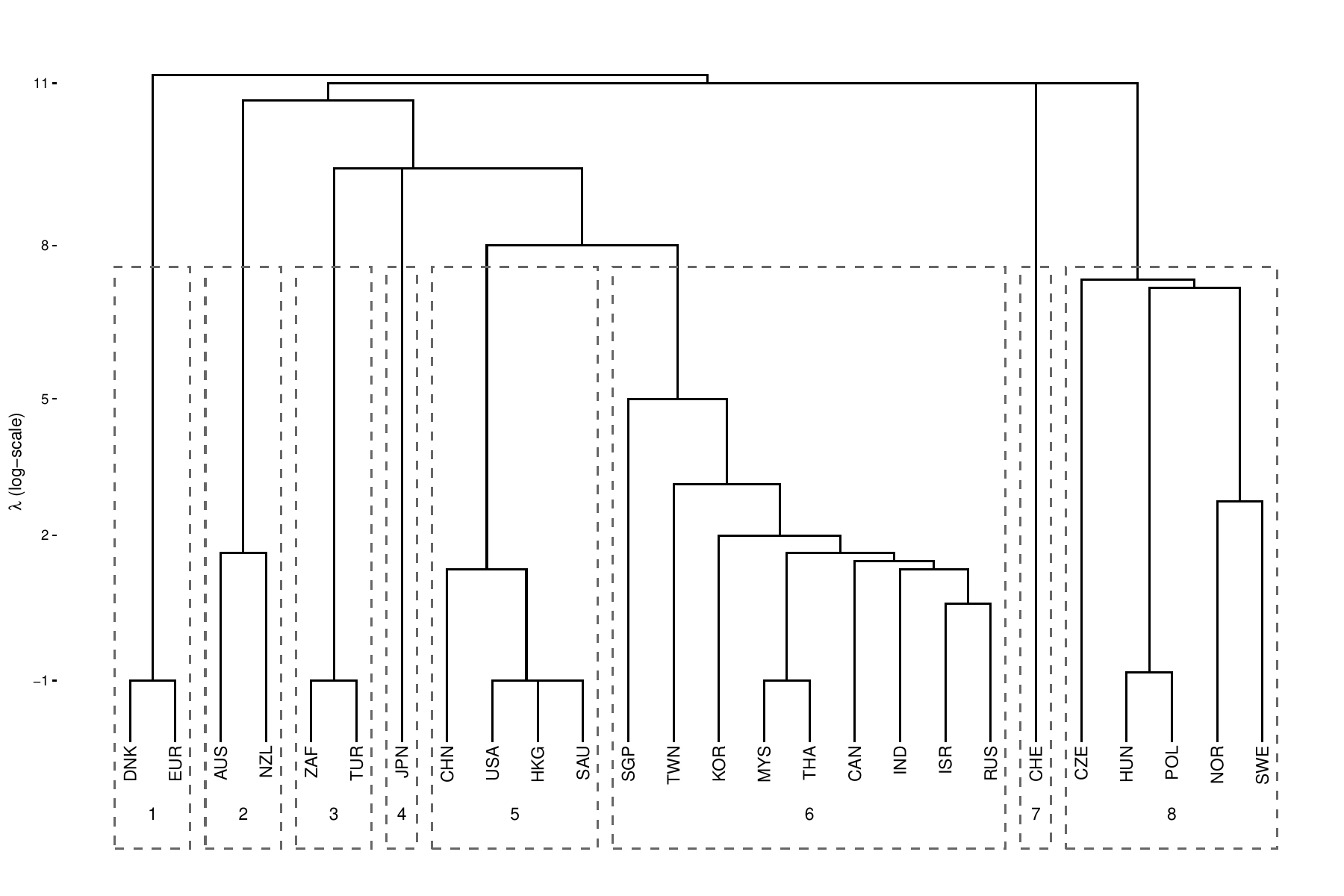}
    \caption{Dendrogram of the hierarchical clustering of daily spot foreign exchange rates against the British pound sterling. The rectangles show an interpretable clustering of the variables with $K=8$. See Table 2 in the Supplementary Materials \ref{appendix:table.currency} for country codes details.}
    \label{fig:currency}
\end{figure}

We propose to examine the results when the cut tree yields $K=8$ clusters. Several of these exhibit meaningful geographical or economic patterns. The largest cluster (shown as group 6) consists predominantly of South-East Asian currencies (SGP, TWN, IND, KOR, THA and MYS), suggesting a common extremal dependence structure within this region. Another cluster (shown as group 5) groups together China (CHN), Hong Kong (HKG), and the United States, three major international financial centres whose currencies exhibit similar extremal interactions. Japan (JPY) and Switzerland (CHE) forms single cluster, indicating a distinct dependence pattern. Finally, two groups of European currencies is recovered as group 8 (CZE, HUN, POL, NOR and SWE) and group 1 (EUR and DNK), consistently with the findings of \cite{wan_zhou}.

Table 3 in the Supplementary materials \ref{appendix:table.application} shows the extremal correlation coefficient between and within the clusters. For instance, it reveals that the group 1 (EUR and DNK) is characterised by a strong inner extremal dependence, and a higher dependence with Switzerland (denoted by group 7). The group 5 and 6 react in the same way towards other currencies but they differ from each other with their internal dependence.

Overall, these results illustrate that the proposed HRC procedure identifies clusters that are not only statistically coherent but also economically interpretable.
Beyond improving estimation through structural regularisation, the proposed methodology provides an interpretable representation of extremal dependence by identifying groups of variables sharing similar interaction patterns.

%% file: sections/conclusion.tex
\section{Conclusion}

The paper introduces a new regularisation framework for estimating HR precision matrices under a block-structure assumption. Rather than exploiting sparsity, the proposed approach promotes homogeneous groups of coefficients through a convex fusion penalty, leading to the simultaneous estimation of both the HR precision matrix and its latent partition. The resulting optimisation problem remains convex and can be solved efficiently using an adaptation of the Clusterpath algorithm originally developed for Gaussian graphical models.

From a theoretical perspective, we have established non-asymptotic concentration bounds for the empirical fusion weights and derived consistency guarantees for both HR precision matrix estimation and partition recovery. To the best of our knowledge, these constitute the first statistical guarantees for fusion-based regularisation in the context of Hüsler–Reiss models. The numerical experiments confirm the practical relevance of the proposed methodology. Across a broad range of balanced, unbalanced, and perturbed block configurations, the HRC procedure accurately recovers the latent partition while providing a satisfactory estimate of the HR precision matrix. These results illustrate that exploiting block homogeneity can substantially improve structural recovery without compromising estimation accuracy.
Beyond its statistical performance, the proposed framework provides an interpretable representation of extremal dependence through groups of variables sharing similar interaction patterns. In this sense, it offers a complementary alternative to sparsity-based regularisation for high-dimensional inference in multivariate extremes.

Several directions deserve further investigation. A first challenge concerns the automatic selection of the regularisation parameter, for instance, through information criteria or stability-based methods. Another promising direction is to relax the exact block assumption by allowing approximately homogeneous groups or hierarchical block structures, which would further broaden the range of practical applications. Finally, although our methodology has been developed for Hüsler–Reiss models, the underlying idea of fusion-based regularisation is considerably more general. Whenever a suitable precision-type parameterisation is available, similar techniques could be developed for other classes of multivariate extreme-value models. More broadly, we believe that this work demonstrates the potential of structural regularisation beyond sparsity and opens the way towards a richer class of parsimonious models for high-dimensional extreme-value statistics.

%% file: sections/appendix.tex
\newpage
\appendix
\label{appendix}

\section{Hüsler-Reiss Clusterpath algorithm}
\label{app:subsec:HR_clusterpath_algo}

The HRC procedure is detailed by the pseudo-code in Algorithm \ref{hrcp-algo}. The function $cluster\_fusion\_step$ and $block\_gradient\_step$ refer to the steps described respectively in the subsections 3.2.1 and 3.2.2.

\begin{algorithm}[h]
\begin{algorithmic} 
\REQUIRE Variogram estimated $\hat \Gamma$, Initial estimates $R^{(0)}$ and $clusters^{(0)}$ weights $W = (w_{ij})_{i,j \in V}$, regularisation $\lambda$, $\varepsilon_f$ fusion threshold, $\varepsilon_g$ the convergence threshold, $t_{max}$ maximal number of iteration.
\ENSURE
\STATE $l^{(0)} \leftarrow L_{\mathcal P}(R^{(0)},~ clusters^{(0)}, \lambda)$
\STATE $l^{(1)} \leftarrow 2 l^{(0)}$
\STATE $K_{max} \leftarrow nrow(R^{(0)})$
\STATE $t \leftarrow 0$
\WHILE{$|l^{(t+1)} / l^{(t)} - 1| > \varepsilon_c$ and $t \le t_{max}$}
\STATE $t \leftarrow t+1$
\STATE $R^{(t)} \leftarrow R^{(t-1)}$
\STATE $cluster^{(t)} \leftarrow cluster^{(t-1)}$
\STATE $(k^*,\ell^*) \leftarrow \arg\min_{k,\ell} \tilde D(\vect{r}^{(t)}_{k\cdot}, \vect{r}^{(t)}_{\ell\cdot})$
\WHILE{$\tilde D(\vect{r}^{(t)}_{k^*\cdot}, \vect{r}^{(t)}_{\ell^*\cdot}) \le \varepsilon_f$}
\STATE $R^{(t)},~ clusters^{(t)} \leftarrow cluster\_fusion\_step(R^{(t)},~ clusters^{(t-1)}, k^*,\ell^*)$
\STATE $K_{max} \leftarrow K_{max} - 1$
\STATE $(k^*,\ell^*) \leftarrow \arg\min_{k,\ell} \tilde D(\vect{r}^{(t)}_{k\cdot}, \vect{r}^{(t)}_{\ell\cdot})$
\ENDWHILE
\FOR{$k = 1, \dots, K_{max}$}
\STATE $R^{(t)} \leftarrow block\_gradient\_step(R^{(t-1)}, clusters^{(t)} , k)$
\ENDFOR
\STATE $l^{(t+1)} \leftarrow L_{\mathcal P}(R^{(t)},~ clusters^{(t)}, \lambda)$
\ENDWHILE
\STATE $\hat R \leftarrow R^{(t)}$
\STATE $clusters_{f} \leftarrow clusters^{(t)}$
\RETURN $\hat R$, $clusters_f$
\end{algorithmic}
\caption{Pseudo algorithm of HRCP with fixed $\lambda$.}
\label{hrcp-algo}
\end{algorithm}

\section{Derivative computation}
\label{derivative-sect}

The function $L$ in Equation \eqref{nllh} involves a generalised determinant, making the computation of the gradient challenging, as it entails evaluating a computationally expensive Moore-Penrose inverse of a matrix. To bypass this computational burden, our procedure relies on an equivalent formulation of $L$, whose gradient evaluation require only a single matrix inversion. This version introduced by \citet[Appendix C]{engelke2025extremal}, is defined by
\begin{equation*}
    \tilde L(\Theta) = - \log (\det(P^\top \Theta P) )- \frac 12 \text{tr}(\Theta PP^\top \Gamma^\star P P^\top), 
    \quad \text{for all } \Theta \in <\vect{1}_d \vect{1}_d^\top>^\perp,
\end{equation*}
with $P$ the full rank matrix such that $PP^\top = \Pi$, the projection matrix of the space $<\vect{1}_d>^\perp$. Let $\tilde L_\pen$ denote the penalised version of $\tilde L$. Computing each block gradient and block Hessian of this alternative objective function is straightforward but requires meticulous care. The results are given below.

We first recall that, at this stage of the algorithm, $K'$ clusters have been identified and the current matrix $\Theta$, associated with the reduced matrix $R$, is written as 
\begin{equation*}
  \Theta = \begin{pmatrix} 
(a_{1} - r_{11})I_{d_1} & 0 & \dots & 0 \\
0 & (a_{2} - r_{22})I_{d_2} & \dots & 0 \\
\vdots & \vdots & \ddots & \vdots \\
0 & 0 & \dots &(a_{K'} - r_{K'K'})I_{d_K'}
\end{pmatrix} + U R U^\top,
\end{equation*}
where $U$ is the cluster matrix defined in Equation \eqref{eq:block_structure_Theta}, adapted to $K'$ clusters.

Let $m \in \{1, \dots, K'\}$. With this configuration, the computation of the gradient and the Hessian for the  $m$-th block involves only the entries of the corresponding column in the reduced matrix $R$. Hence, the $m$-th block gradient is
\begin{equation*}
    \nabla_m \tilde L_\pen(\Theta, \lambda) = \left(\frac{\partial \tilde L_\pen(\Theta, \lambda)}{\partial r_{mk}}\right)_{k = 1,\dots, K'}\, ,
\end{equation*}
and the $m$-th block hessian is
\begin{equation*}
    \nabla^2_m \tilde L_\pen(\Theta, \lambda) = \left( \frac{\partial^2 \tilde L_\pen(\Theta, \lambda)}{\partial r_{mk} \partial r_{m\ell}}\right)_{k,\ell = 1, \dots, K'}\, .
\end{equation*}
These quantities are computed by calculating the gradient and Hessian of $\tilde L$ and $\pen$.

\subsection{Block gradient of \texorpdfstring{$\tilde L(\Theta)$}{L(Theta)}}\label{subsec:gradientL}

To shorten notation, let $M_\Theta = P(P\Theta P^\top)^{-1} P^\top$ and $\hat\Gamma_P = PP^\top \hat\Gamma PP^\top$. We obtain
\begin{align*}
    & \frac{\partial \tilde L(\Theta)}{\partial r_{mk}} = d_m \sum_{i \in C_k} (M_\Theta + \frac 12 \hat\Gamma_P)_{ii} + d_k \sum_{i \in C_m} (M_\Theta + \frac 12 \hat\Gamma_P)_{ii} - 2 (U^\top M_\Theta U)_{mk} - (U^\top \hat\Gamma_P U)_{mk}, \\
    & \frac{\partial \tilde L(\Theta)}{\partial r_{mm}} = d_m \sum_{i \in C_m} (M_\Theta + \frac 12 \hat\Gamma_P)_{ii} - (U^\top M_\Theta U)_{mm} - \frac 12 (U^\top \hat\Gamma_P U)_{mm}.
\end{align*}

\subsection{Block gradient of \texorpdfstring{$\pen(\Theta)$}{P(Theta)} }\label{subsec:gradientP}

We stress that

\begin{equation*}
    \pen(\Theta) = \sum_{k<\ell} W_{k\ell} \tilde D^2(\vect{r}_{k\cdot}, \vect{r}_{\ell\cdot})\,,
\end{equation*}
where $W_{k\ell} = \sum_{i \in C_k} \sum_{j \in C_\ell} w_{ij}$.
\noindent Hence, we only need to compute the derivative of $\tilde D^2$ for each coefficient $r_{k\ell}$. Let $k,\ell \in V$. For $k\neq m$ and $q\in V$, we have
\begin{equation*}
    \frac{\partial \tilde D^2(\vect{r}_{m\cdot }, \vect{r}_{k\cdot })}{\partial r_{mq}} =
    \begin{cases} 
        2d_q(r_{mq} -r_{kq}) & \text{if } k,m \neq q,\\ 
        2(d_m-1)(r_{mm} -r_{mk}) & \text{if } q = m,\\ 
        2(d_m-1)(r_{mq} -r_{mm}) +2(d_q-1)(r_{mq} -r_{qq}) & \text{if } k = q,
    \end{cases}
\end{equation*}
and for $k, \ell \neq m$,
\begin{equation*}
\frac{\partial \tilde D^2(\vect{r}_{k\cdot}, \vect{r}_{\ell \cdot })}{\partial r_{mq}} =
\begin{cases} 
2d_m(r_{km} -r_{k\ell}) & \text{if } k = q,\\ 
2d_m(r_{\ell m} -r_{k\ell}) & \text{if } \ell = q,\\ 
0 & \text{otherwise.}
\end{cases}
\end{equation*}

\noindent We can then deduce the derivatives using the expression 
\begin{align*}
    \frac{\partial \pen(\Theta)}{\partial r_{mq}} & = \sum_{k<\ell} W_{k\ell} \frac{\partial \tilde D^2(\vect{r}_{k\cdot }, \vect{r}_{\ell\cdot })}{\partial r_{mq}} \\ 
    & =\sum_{k\neq m} W_{km} \frac{\partial \tilde D^2(\vect{r}_{m\cdot }, \vect{r}_{k\cdot })}{\partial r_{mq}} + \sum_{\underset{k,\ell \neq m }{k<\ell}} W_{kl} \frac{\partial \tilde D^2(\vect{r}_{k\cdot }, \vect{r}_{\ell\cdot })}{\partial r_{mq}}.
\end{align*}

\subsection{Block Hessian of \texorpdfstring{$\tilde L(\Theta)$}{L(Theta)}}

We define the following matrices
\begin{equation*}
    (E_{k\ell})_{ij} = \begin{cases}
    1& \text{if } (i,j)\in \{(k,\ell), (\ell,k) \}, \\
    0 & \text{otherwise}
    \end{cases} \quad \text{and} \quad 
    (B_{k\ell})_{ij} = \begin{cases}
    -d_k & \text{if } i=j \in C_\ell,\\
    -d_\ell & \text{if } i=j \in C_k,\\
    0 & \text{otherwise}.
    \end{cases}
\end{equation*}

\noindent Using the gradient derived in Subsection \ref{subsec:gradientL}, we obtain
\begin{align*}
    & \frac{\partial^2 \tilde L(\Theta)}{\partial r_{mk} \partial r_{m\ell}} = d_m \sum_{i \in C_k} (\mathcal M_{m\ell})_{ii} + d_k \sum_{i \in C_m} (\mathcal M_{m\ell})_{ii} - 2(U^\top \mathcal M_{m\ell} U)_{mk}, \\
    & \frac{\partial^2 \tilde L(\Theta)}{\partial r_{mm} \partial r_{mk}} = d_m \sum_{i \in C_m} (\mathcal M_{mk})_{ii} - (U^\top \mathcal M_{mk} U)_{mm}, \\
    & \frac{\partial^2 \tilde L(\Theta)}{\partial r_{mm}^2} = d_m \sum_{i \in C_m} (\mathcal M_{mm})_{ii} - (U^\top \mathcal M_{mm} U)_{mm},
\end{align*}
with $\mathcal M_{k\ell} = - M_\Theta(B_{k\ell} + U E_{k\ell} U^\top) M_\Theta$.

\subsection{Block Hessian of \texorpdfstring{$\pen(\Theta)$}{P(Theta)}}

Using the gradient of the penalty computed in Subsection \ref{subsec:gradientP}, we obtain
\begin{align*}
    & \frac{\partial^2 \pen(\Theta)}{\partial r_{mk} \partial r_{m\ell}} = - 2 W_{k\ell} d_m, \\
    & \frac{\partial^2 \pen(\Theta)}{\partial r_{mk} \partial r_{mk}} = 2 d_k \sum_{q\neq m,k} W_{mq} + 2(d_m+ d_k -2) W_{km} + 2 d_m \sum_{q\neq m,k} W_{kq}, \\
    & \frac{\partial^2 \pen(\Theta)}{\partial r_{mm} \partial r_{mk}} = - 2(d_m-1) W_{km}, \\
    &  \frac{\partial^2 \pen(\Theta)}{\partial r_{mm}^2} = 2(d_m-1) \sum_{k\neq m} W_{km}.
\end{align*}

\section{Illustrative Hüsler-Reiss block model}

\subsection{Example of block structured HR precision matrix}
\label{example_block_structure_Theta}
 
The following example illustrates the block representation introduced in Subsection 2.4 and will serve as a running example in the numerical experiments of Section \ref{sect-sim}.

We consider the matrix $\Theta^\star \in \SPd$ given by
    \begin{equation}
    \label{eq:matrix_example}
    \footnotesize{\Theta^\star = \begin{pmatrix}
          2 & -0.5 & -0.5 & -0.5 & -0.25 & -0.25 & 0 & 0 & 0 \\
          -0.5 & 2 & -0.5 & -0.5 & -0.25 & -0.25 & 0 & 0 & 0 \\ 
          -0.5 & -0.5 & 2 & -0.5 & -0.25 & -0.25 & 0 & 0 & 0 \\
          -0.5 & -0.5 & -0.5 & 2 & -0.25 & -0.25 & 0 & 0 & 0 \\
          -0.25 & -0.25 & -0.25 & -0.25 & 2.75 & -0.75 & -0.25 & -0.25 & -0.25 \\
          -0.25 & -0.25 & -0.25 & -0.25 & -0.75 & 2.75 & -0.25 & -0.25 & -0.25 \\
          0 & 0 & 0 & 0 & -0.25 & -0.25 & 1.5 & -0.5 & -0.5 \\
          0 & 0 & 0 & 0 & -0.25 & -0.25 & -0.5 & 1.5 & -0.5 \\
          0 & 0 & 0 & 0 & -0.25 & -0.25 & -0.5 & -0.5 & 1.5 \\
        \end{pmatrix}\,.}
  \end{equation}
This matrix satisfies the block structure \eqref{eq:block_structure_Theta} with $K=3$, and a partition $\mathcal{C}$ formed by the clusters $C_1 = \{1,2,3,4\}$, $C_2 = \{5,6\}$, and $C_3 = \{7,8,9\}$, with size $d_1 = 4$, $d_2 =2$, and $d_3=3$ respectively. We also have $a_1=2$, $a_2 = 2.25$, and $a_3 = 1.5$ and the matrices $R$ and $U$ are given by
  \begin{equation*}
  \footnotesize{
    R = \begin{pmatrix}
          -0.5 & -0.25 & 0 \\
          -0.25 & -0.75 & -0.25 \\
          0 & -0.25 & -0.5
        \end{pmatrix}
        \quad \text{and} \quad
        U^\top = \begin{pmatrix}
1 & 1 & 1 & 1 & 0 & 0 & 0 & 0 & 0 \\
0 & 0 & 0 & 0 & 1 & 1 & 0 & 0 & 0 \\
0 & 0 & 0 & 0 & 0 & 0 & 1 & 1 & 1
\end{pmatrix}\,.}
    \end{equation*}

An important consequence of the reduced representation concerns graphical interpretation. The nullity of a coefficient $r_{ij}$ in the reduced matrix $R$ (as for instance the bottom left one in \eqref{eq:matrix_example}) indicates conditional extremal independence between the two clusters $i$ and $j$ conditionally on the others. This is a consequence of the global Markov property of extremal graphical models \citep{engelkeGraphicalModelsExtremes2020}. Hence, a graph with $K$ nodes induced by the coefficients of $R$ can be used with vertices corresponding to a cluster $C_k$ and the edge the conditional block dependence. Figure \ref{fig:graph_block} illustrates the interpretation of the block structure in terms of graph for the HR precision matrix \eqref{eq:matrix_example}.

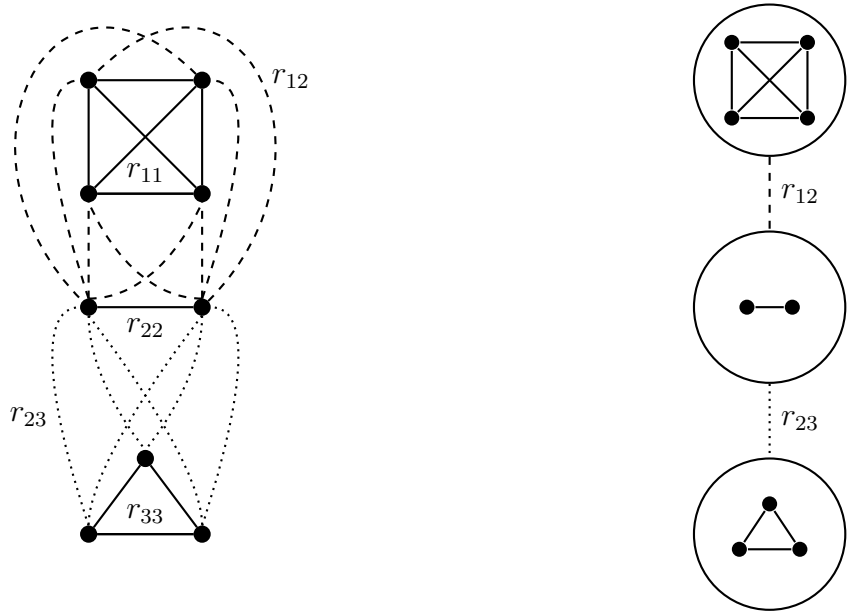
\begin{figure}[!h]
  \centering
  \input{figures/block-graph.tex}
  \caption{Graphical interpretation of the HR precision matrix \eqref{eq:matrix_example}. Left: extremal graphical model associated with this matrix. Right: reduced graphical representation obtained from the block structure. The dashed and dotted edges are weighted by $r_{12}$ and $r_{23}$, respectively. The missing edge between the first and third clusters corresponds to the null coefficient $r_{13}=0$.}
  \label{fig:graph_block}
\end{figure}

\subsection{Characterisation of HR precision matrix}
\label{appendix:characterisation.hr}

The following proposition provides simple sufficient conditions ensuring that the representation \eqref{eq:block_structure_Theta} belongs to $\SPd$.

\begin{proposition}
  \label{prop:valid-r}
  Let $\Theta$ be a matrix satisfying the block structure \eqref{eq:block_structure_Theta}
  and let $R$ be the associated reduced matrix. We define $B = \text{diag}((d_k)_{k=1,\dots, K})$ and $T=\text{diag}((\sum_{\ell=1}^K d_\ell r_{k\ell})_{k=1, \dots, K}) $. If the matrix $\tilde{R} = RB - T \in \mathcal{SP}^1_K$ (the $K$-dimensional analogue of $\SPd$) and $a_k - r_{kk} > 0$ for all $k \in \{1,\dots, K\}$, then $\Theta \in \SPd$.
\end{proposition}

\begin{proof}
Since $\Theta$ is assumed to satisfy the block structure \eqref{eq:block_structure_Theta}, it is sufficient to prove that it is positive semi-definite with $\text{rank}(\Theta) = d-1$ to guarantee that it belongs to $\SPd$. This is equivalent to showing that the matrix has a simple zero eigenvalue, while all remaining eigenvalues are strictly positive. Given this block structure, the characteristic polynomial $P_\Theta(X) = \det(\Theta - X I_d)$ of $\Theta$ can be written as
\begin{equation*}
P_\Theta(X)
= \begin{vmatrix}
    A_1 - XI_{d_1} & r_{12} \vect{1}_{d_1} \vect{1}^\top_{d_2}& \cdots & \cdots & r_{1K} \vect{1}_{d_1} \vect{1}^\top_{d_K}\\
    r_{21} \vect{1}_{d_2} \vect{1}^\top_{d_1} & A_2 - XI_{d_2} & \ddots &  &  \vdots \\
    \vdots & \ddots & \ddots & \ddots &\vdots \\
    \vdots & & \ddots & A_{K-1}- XI_{d_{K-1}}&  r_{K(K-1)} \vect{1}_{d_{K-1}} \vect{1}^\top_{d_K} \\
    r_{K1} \vect{1}_{d_K} \vect{1}^\top_{d_1} & \cdots & \cdots &r_{K(K-1)} \vect{1}_{d_K} \vect{1}^\top_{d_{K-1}}  & A_K - XI_{d_K}
\end{vmatrix}\,,
\end{equation*}
with
\begin{equation*}
A_k =
\begin{pmatrix}
a_k & r_{kk} & \cdots& r_{kk} \\
r_{kk} & a_k & \ddots & \vdots \\
\vdots & \ddots& \ddots & r_{kk} \\
r_{kk} & \cdots & r_{kk} & a_{k}
\end{pmatrix} \in \mathcal{M}_{d_k}(\R)\,,
\quad k = 1, \ldots, K\,.
\end{equation*}
We focus on the first $d_1$ rows and subtract the $d_1$-th row from the preceding ones allowing to factorise the determinant by $a_1 - r_{11} - X$. The $d_1$ first rows become
\begin{equation*}
(a_1 - r_{11} - X)^{d_1-1}
\begin{array}{c}
\left|
\begin{array}{cccccccccccc}
1 & 0 & \cdots & 0 & -1 &
0 & \cdots & 0 &
\cdots & 0 & \cdots & 0 \\

0 & 1 & \ddots & \vdots & \vdots &
0 & \cdots & 0 &
\cdots & 0 & \cdots & 0 \\

\vdots & \ddots & \ddots & 0 & -1 &
0 & \cdots & 0 &
\cdots & 0 & \cdots & 0 \\

0 & \cdots & 0 & 1 & -1 &
0 & \cdots & 0 &
\cdots & 0 & \cdots & 0 \\

r_{11} & \cdots & \cdots & r_{11} & a_1-X &
r_{12} & \cdots & r_{12} &
\cdots &
r_{1K} & \cdots & r_{1K}
\end{array}
\right|
\\[-0.3em]
\begin{array}{cccc}
\underbrace{\hspace{4.8cm}}_{d_1}
&
\underbrace{\hspace{2.5cm}}_{d_2}
&
\hspace{0.4cm}
&
\underbrace{\hspace{2.5cm}}_{d_K}
\end{array}
\end{array}.
\end{equation*}
By summing the first $d_1-1$ columns with the $d_1$-th column, we then obtain
\begin{equation*}
(a_1 - r_{11} - X)^{d_1-1}
\begin{array}{c}
\left|
\begin{array}{cccccccccccc}
1 & 0 & \cdots & 0 & 0 &
0 & \cdots & 0 &
\cdots & 0 & \cdots & 0 \\

0 & 1 & \ddots & \vdots & \vdots &
0 & \cdots & 0 &
\cdots & 0 & \cdots & 0 \\

\vdots & \ddots & \ddots & 0 & 0 &
0 & \cdots & 0 &
\cdots & 0 & \cdots & 0 \\

0 & \cdots & 0 & 1 & 0 &
0 & \cdots & 0 &
\cdots & 0 & \cdots & 0 \\

r_{11} & \cdots & \cdots & r_{11} & a_1+ (d_1-1)r_{11}-X &
r_{12} & \cdots & r_{12} &
\cdots &
r_{1K} & \cdots & r_{1K}
\end{array}
\right|
\\[-0.3em]
\begin{array}{cccc}
\underbrace{\hspace{7.1cm}}_{d_1}
&
\underbrace{\hspace{2.5cm}}_{d_2}
&
\hspace{0.4cm}
&
\underbrace{\hspace{2.5cm}}_{d_K}
\end{array}
\end{array},
\end{equation*}
which leads, by expending the determinant along the first $d_1-1$ rows, to 
\begin{equation*}
P_\Theta(X)
= (a_1 - r_{11} - X)^{d_1-1}\begin{vmatrix}
    a_1+ (d_1-1)r_{11}-X  & r_{12}  \vect{1}^\top_{d_2}& \cdots &  r_{1K} \vect{1}^\top_{d_K}\\
    d_1 r_{21} \vect{1}_{d_2} & A_2 - XI_{d_2} & \ddots &  \vdots \\
    \vdots & \ddots  & \ddots &\vdots \\
    d_1 r_{K1} \vect{1}_{d_K} & \cdots & \cdots & A_K - XI_{d_K}
\end{vmatrix}\,.
\end{equation*}
After repeating this operation on the other blocks of the matrix, we obtain
\begin{equation*}
P_\Theta(X) = \left(\prod_{k = 1} ^K (a_k - r_{kk} - X)^{d_k-1}\right) \det(\tilde R - X I_K)\,.
\end{equation*}
Since $\tilde{R}$ is symmetric and satisfies $\tilde{R}\mathbf{1}_K = \mathbf{0}$ by construction, having $\tilde{R} \in \mathcal{SP}^1_K$ and $a_k > r_{kk}$ for all $k \in V$ is sufficient to guarantee that all eigenvalues of $\Theta$ are strictly positive except for a single simple zero eigenvalue.

\end{proof}

Proposition \ref{prop:valid-r} provides a simple validity check for the estimators introduced in Section \ref{method} and is also used to generate valid block-structured HR precision matrices when the clusters are of equal size, see next section.

\paragraph{Generating a block-structured matrix \texorpdfstring{$\Theta$}{Theta} in \texorpdfstring{$\SPd$}{SPd1}}
\label{appendix:note-valid}

Proposition \ref{prop:valid-r} provides a practical way to generate valid block-structured HR precision matrices with equal-sized clusters $p$ by ensuring the construction of an admissible reduced matrix $R$. This stems from two main observations made in this balanced-cluster scenario:
\begin{enumerate}[label=(\alph*)]
    \item Since all the cluster sizes $d_\ell=p$, according to the definition of $a_k$ in Equation \eqref{eq:block_structure_Theta}, the condition $a_k-r_{kk} > 0$ simplifies to $\sum_{\ell=1}^K r_{k\ell}<0$.
    \item Similarly, the matrix $\tilde R$ takes the simpler form
    $$
    \tilde R =  p (R - \tilde T),
    $$
   where $\tilde T := \text{diag}((\sum_{\ell=1}^K r_{k\ell})_{k=1, \dots, K})$, and consequently satisfies $\tilde R \vect{1}_K = 0$. Hence, $\tilde R \in \mathcal{SP}^1_K$ and can be easily constructed as implemented for instance in the R package \cite{graphicalExtremes}.
\end{enumerate}
Thus, a block-structured precision matrix $\Theta$ with balanced clusters of arbitrary size can be generated by constructing the corresponding reduced matrix ${R}$ as follows: 

\paragraph{Step 1:} Get a matrix $\tilde R$ in $\mathcal{SP}_K^1$.
\paragraph{Step 2:} For the off-diagonal coefficients $r_{k\ell}$ of $R$, set $r_{k\ell} = \tilde r_{k\ell}$
\paragraph{Step 3:} The diagonal entries $\tilde{r}_{kk}$ are then equal to $-\sum_{\ell\neq k} r_{k\ell}$. To fulfil the condition $\sum_{\ell=1}^K r_{k\ell} < 0$, each diagonal element $r_{kk}$ is simply chosen so that $r_{kk} < \tilde r_{kk}$

\section{Proofs}
\label{appendix:proof}

This section provides the proofs of the main results presented in the paper.

\subsection{Proof of Proposition \ref{path-sol}}

Assume that there exists $\tilde \lambda>0$ such that $\pen(\hat \Theta(\tilde \lambda)) = 0$. For any $\lambda \ge \tilde \lambda$, it holds
\begin{equation*}
0\leq \pen(\hat\Theta(\lambda)) \le \pen(\hat\Theta(\tilde \lambda)) = 0
\end{equation*}
and consequently, by strict convexity of the function $L$ defined in \eqref{nllh}, $\hat \Theta(\lambda) =\hat \Theta(\tilde \lambda)$. This motivates considering the mapping $\lambda \mapsto \hat \Theta(\lambda)$ only on the interval $[0, \tilde \lambda]$. Before proving the continuity of this solution path, it is helpful to establish the following lemmas.

\begin{lemma}
    \label{coercivity}
    $L$ is coercive, \textit{i.e.} $L(\Theta) \to +\infty$ as $||\Theta||_F \rightarrow +\infty$.
\end{lemma} 

\begin{proof}[Proof of Lemma \ref{coercivity}]
    Let $\Theta \in \SPd$. By definition, $\Theta$ has $d$ eigenvalues $0=\lambda_1 < \lambda_2\leq \dots\leq \lambda_d$. Its symmetry yields the following spectral decomposition
     \begin{equation}
     \label{eq:decomposition}
         \Theta = Q \Lambda Q^\top\,,
    \end{equation}
     with $\Lambda= \text{diag}((\lambda_k)_{k=1,\dots,d})$, and $Q$ is an orthogonal matrix whose $k$-th column corresponds to an eigenvector of $\Theta$ associated with $\lambda_k$. In particular, since $\Theta \mathbf{1}_d = \mathbf{0}_d$, its first column $\vect{Q}_{\cdot 1}$ is proportional to the vector $\mathbf{1}_d$. It follows that
       \begin{equation*}
        ||\Theta||_F^2 = \sum_{i=2}^d \lambda_i^2\,.
    \end{equation*}
    Hence, the maximum eigenvalue $\lambda_d$ necessarily diverges to infinity as $\|\Theta\|_F \to \infty$. Assume for now that it is the unique eigenvalue doing so. Using \eqref{eq:decomposition}, the function $L$ can be written as
    \begin{align}
    L(\Theta) & = -\sum_{k=2}^{d} \log(\lambda_k)- \frac 12 \text{tr}(\Gamma Q \Lambda Q^\top)\,  \nonumber \\
    & = \left(-\log(\lambda_d) - \frac 12 \left(Q^\top \Gamma Q\right)_{dd} \lambda_d \right) +\sum_{k=2}^{d-1} \left(-\log(\lambda_k)- \frac 12 (Q^\top \Gamma Q)_{kk} \lambda_k  \right).\label{eq:Ldecompo}
\end{align}
 By definition, the variogram matrix $\Gamma$ satisfies $a^\top \Gamma a <0$ for any non-zero vector $a$ such that $a_1 + \dots +a_d = 0$. The orthogonality condition $Q_{\cdot d}^\top \mathbf{1}_d = 0$ implies $(Q^\top \Gamma Q)_{dd} < 0$, and consequently the first term in Equation \eqref{eq:Ldecompo} diverges to $+\infty$ as $||\Theta||_F \to \infty$. Since the second term in Equation \eqref{eq:Ldecompo} is bounded from below as $\|\Theta\|_F \to  \infty$, it follows that $L(\Theta) \to +\infty$ as $\|\Theta\|_F \to  \infty$, which establishes the coercivity of $L$. By similar reasoning, this result extends to the case where several eigenvalues diverge to infinity as $\|\Theta\|_F \to \infty$.
\end{proof}

\begin{lemma}
    \label{bound-min}
Let $\{\hat \Theta(\lambda)\}_{\lambda \in [0, \tilde \lambda]}$ be the set of solutions to the optimisation problem defined in \eqref{prob-nllh} for $\lambda \in [0, \tilde \lambda]$. The set $E = \{||\hat \Theta(\lambda)||_F, \,\lambda \in [0, \tilde \lambda]\}$ is bounded.
\end{lemma} 

\begin{proof}[Proof of Lemma \ref{bound-min}] By definition, we have for all $\lambda \le \tilde \lambda$: 
\begin{equation*}
    L(\hat \Theta(\lambda)) \leq  L(\hat \Theta(\lambda)) + \lambda \mathcal P(\hat \Theta(\lambda)) \le L(\hat \Theta(\tilde \lambda)) + \lambda \mathcal P(\hat \Theta(\tilde \lambda)) = L(\hat \Theta(\tilde \lambda)).
\end{equation*}
Hence, the set $\{L(\hat{\Theta}(\lambda))\}_{\lambda \in [0, \tilde{\lambda}]}$ is bounded. Consequently, $E$ must also be bounded, as an unbounded $E$ would contradict the coercivity of $L$ established in Lemma \ref{coercivity}.
\end{proof}

\begin{proof}[Proof of Proposition \ref{path-sol}] Let $(\lambda_n)_{n \in \N}$ be a sequence in $[0, \tilde \lambda]$ converging to $\lambda'$, and consider the associated sequence $(\hat\Theta_n)_{n \in \N}$ in $\SPd$ defined as
\begin{equation*}
    \hat{\Theta}_n = \hat \Theta(\lambda_n).
\end{equation*}
We aim to show that $(\hat \Theta_n)_{n \in \mathbb N}$ $(\Theta_n)_{n \in \mathbb N}$ converges to $\Theta':= \hat \Theta(\lambda')$

By Lemma \ref{bound-min}, this sequence is bounded, and thus, there exists a convergent subsequence $(\hat\Theta_{\phi(n)})_{n \in \mathbb N}$ whose limit is denoted by $\Theta''$. Since $\hat \Theta_{\phi(n)}$ minimises the objective function at $\lambda_{\phi(n)}$, it holds for all $\Theta \in \mathcal{S}_d^1$ that
\begin{equation*}
    L(\hat\Theta_{\phi(n)}) + \lambda_{\phi(n)} \mathcal P(\hat\Theta_{\phi(n)}) \le L(\Theta) + \lambda_{\phi(n)} \mathcal P(\Theta).
\end{equation*}
Hence, taking the limit as $n \rightarrow \infty$ yields
\begin{equation*}
L(\Theta'') + \lambda' \mathcal P(\Theta'') \le L(\Theta) + \lambda' \mathcal P(\Theta).
\end{equation*}
This implies that $\Theta''$ is a solution of the minimisation problem \eqref{prob-nllh}, and by uniqueness $\Theta'' = \Theta'$. 

Since the sequence $(\Theta_n)_{n \in \mathbb N}$ is bounded and admits $\hat \Theta(\lambda')$ as a unique subsequential limit, it converges to $\hat \Theta(\lambda')$. Hence, the mapping $\lambda \mapsto \hat \Theta(\lambda)$ is continuous.
\end{proof}

\subsection{Proof of Proposition \ref{prop:merging_step}}

Before proving that $\Theta_{\text{new}} \in \SPd$, we first establish the following lemma bounding the distance between $\Theta$ and $\Theta_{\text{new}}$.

\begin{lemma}
    \label{lemma:distnewTheta} Let $\Theta \in \SPd$ with associated clusters $C_1, \dots, C_{K'}$ and $\Theta_{new}$ be the matrix obtained after the merging step described in Equation \eqref{r-approx}. Then if $D(C_k, C_\ell) \le \varepsilon_f$, there exists $c>0$ such that
    \begin{equation*}
            ||\Theta - \Theta_{new} ||_F \le c \varepsilon_f,
    \end{equation*}
\end{lemma} 
\begin{proof}[Proof of Lemma \ref{lemma:distnewTheta}]
   Without loss of generality, suppose that the clusters $C_k$ and $C_\ell$ are the only pair of clusters satisfying $D(C_k, C_\ell) \le \varepsilon_f$. Thus, the matrices $\Theta$ and $\Theta_{\mathrm{new}}$ differ only on the columns (and, by symmetry, the rows) corresponding to indices in $C_k \cup C_\ell$.
   Therefore,
    \begin{align*}
         ||\Theta - \Theta_{new} ||_F^2 =& \sum_{i\in C_k \cup C_l} \sum_{j=1}^d (\Theta_{ij} - (\Theta_{new})_{ij})^2+ \sum_{j\in C_k \cup C_l} \sum_{i=1}^d (\Theta_{ij} - (\Theta_{new})_{ij})^2\\ &- \sum_{i,j \in C_k\cup C_l}(\Theta_{ij} - (\Theta_{new})_{ij})^2 \\
         \le & 2 \sum_{i\in C_k \cup C_l} ||\vect{\Theta}_{i\cdot} - (\vect{\Theta_{new}})_{i\cdot} ||^2
    \end{align*}
    where $||\cdot||$ denotes the euclidean norm. Consequently, showing that $||\vect{\Theta}_{i\cdot} - (\vect{\Theta_{new}})_{i\cdot} ||^2 = O(\varepsilon_f^2)$ for each $i\in C_k \cup C_l$  completes the proof.
    
Suppose that $i \in C_k$. We have
    \begin{equation*}
        ||\vect{\Theta}_{i\cdot} - (\vect{\Theta_{new}})_{i\cdot} ||^2 = ||\vect{\Theta}_{i, -i} - (\vect{\Theta_{new}})_{i, -i} ||^2 + (a_i - (a_{new})_i)^2 \, ,
    \end{equation*}
    where $\vect{\Theta}_{i, -i}$ denotes the $i$-th row of $\Theta$ without the $i$-th coefficient. As the map $\vect{\Theta}_{i, -i} \longmapsto a_i=\sum_{j\neq i}\Theta_{i,j}$ is linear, it is also lipschitzian with some constant $K>0$, which implies 
        \begin{equation*}
        ||\vect{\Theta}_{i\cdot} - (\vect{\Theta_{new}})_{i\cdot} ||^2 \le (1+K^2)||\vect{\Theta}_{i, -i} - (\vect{\Theta_{new}})_{i, -i} ||^2.
    \end{equation*}
   Recall that $D(C_k, C_\ell) \le \varepsilon_f$. We bound $\|\vect{\Theta}_{i, -i} - (\vect{\Theta}_{\mathrm{new}})_{i, -i}\|^2$ by decomposing its terms into three parts:
    \begin{enumerate}[label=(\alph*)]
    \item for any $j\notin C_k \cup C_\ell$, the coefficient $(\Theta_{new})_{ij}$ is only a weighted average of $\Theta_{ij}$ and $\Theta_{i'j}$, with $i' \in C_\ell$. It follows that $\sum_{j\notin C_k \cup C_\ell}(\Theta_{ij} - (\Theta_{new})_{ij})^2 \le D^2(\Theta_{i\cdot}, \Theta_{i'\cdot}) \le \varepsilon_f^2$.
    
        \item For any $j\in C_\ell$, $(\Theta_{new})_{ij} = \Theta_{ij}$.
        
        \item For $j \in C_k \setminus \{i\}$, the remaining $d_k-1$ terms give $(d_k-1)(r_{kk}-r_{kl})^2\le D^2(\Theta_{i\cdot}, \Theta_{i'\cdot})\le \varepsilon_f^2$,  for any $i' \in C_\ell$.
    \end{enumerate}
   The same bounds are obtained when  $i\in C_\ell$, and consequently it follows that 
    \begin{equation*}
        ||\vect{\Theta}_{i\cdot} - (\vect{\Theta_{new}})_{i\cdot} ||^2 \le 2(1+K^2)\varepsilon_f^2,
    \end{equation*}
    which completes the proof.
\end{proof}

\begin{proof}[Proof of Proposition \ref{prop:merging_step}]

Consider the metric space
\begin{equation*}
    \mathcal{S}^1_d = \left\{\Theta \in \mathcal M_d(\R): \Theta^\top = \Theta, \,  \Theta \vect{1}_d = \vect{0}_d\right\}\,
\end{equation*}
equipped with the topology of $(\mathcal M_d(\mathbb R), ||\cdot ||_F)$. 
We define the map $g$ that assigns to each matrix $\Theta \in \mathcal{S}^1_d$ its $d-1$ remaining eigenvalues, setting aside the trivial zero eigenvalue associated with the eigenvector $\mathbf{1}_d$:
\begin{align*}
    g : \mathcal{S}^1_d & \longrightarrow \mathbb R^{d-1} \\
    \Theta &\longmapsto \lambda(\Theta) = (\lambda_i(\Theta))_{i=1, \dots,d-1}.
\end{align*}
This mapping is continuous, and consequently $g^{-1}((\mathbb R_+^*)^{d-1}) = \SPd$ is an open subset of $\mathcal S_d^1$. Hence, it exists $\varepsilon^* >0$ such that
\begin{equation*}
B(\Theta, \varepsilon^*)\cap \mathcal S_d^1 \subset \SPd.
\end{equation*}
By construction, the updated matrix $\Theta_{new}$ belongs to $\mathcal S_d^1$ and satisfies $||\Theta_{\text{new}} - \Theta ||_F \leq c \varepsilon_f$ for some constant $c>0$, see Lemma \ref{lemma:distnewTheta}. Subsequently, choosing $\varepsilon_f < \varepsilon^*/c$ guarantees that $\Theta_{new}\in \SPd$.
\end{proof}

\subsection{Proof of Proposition \ref{prop:w-concentration}}
\label{appendix:w-concentration}
Let $i,j \in V$. We first derive bounds involving $\hat w_{ij}^{(n)}$ by considering two cases, depending on whether $i$ and $j$ belong to the same cluster or to distinct clusters.

\paragraph*{Case 1: $i$ and $j$ are in the same cluster.}

Since $\Gamma^\star_{ik} = \Gamma^\star_{jk}$ for all $k \neq i, j$, the triangle inequality applied to the pseudometric $D$ leads to the following relation: 
\begin{equation*}
     D((\hat{ \vect{\Gamma}}_n)_{i\cdot}, (\hat{ \vect{\Gamma}}_n)_{j\cdot})  \le  D((\hat{ \vect{\Gamma}}_n)_{i\cdot}, \vect{\Gamma}^\star_{i\cdot}) +  D((\hat{ \vect{\Gamma}}_n)_{j\cdot}, \vect{\Gamma}^\star_{j\cdot}) \le 2 (d-2)||\hat \Gamma_n - \Gamma^\star||_\infty.
\end{equation*}
On the event $\{||\hat \Gamma_n - \Gamma^\star||_\infty \le \delta_n\}$, the previous inequality reduces to
\begin{equation*}
 D\left((\hat{ \vect{\Gamma}}_n)_{i\cdot}, (\hat{ \vect{\Gamma}}_n)_{j\cdot}\right)  \le 2(d-2)\delta_n.
\end{equation*}
Let $\zeta >0$. By the bijectivity of the function $t \mapsto \exp(-\zeta t^2)$ on $\mathbb{R}_+$, it follows that
\begin{equation}
\label{eq:w-bound-1}
|1 - \hat w_{ij}^{(n)}| \le 1 - \exp\{-4(d-2)^2\zeta\delta_n^2\}.
\end{equation}

\paragraph{Case 2: $i$ and $j$ are in different clusters.}
In this case, using the definition of $\mu^\star$ and applying the triangle inequality once again yield
\begin{align*}
    \mu^\star & \le D(\vect{\Gamma}^\star_{i\cdot}, \vect{\Gamma}^\star_{j\cdot}) \\
    & \le  D((\hat{ \vect{\Gamma}}_n)_{i\cdot}, \vect{\Gamma}^\star_{i\cdot}) +  D((\hat{ \vect{\Gamma}}_n)_{j\cdot}, \vect{\Gamma}^\star_{j\cdot})+ D((\hat{ \vect{\Gamma}}_n)_{i\cdot}, (\hat{ \vect{\Gamma}}_n)_{j\cdot}).
\end{align*}
On the event $\{||\hat \Gamma_n - \Gamma^\star||_\infty \le \delta_n\}$, this gives
\begin{equation*}
    D((\hat{ \vect{\Gamma}}_n)_{i\cdot}, (\hat{ \vect{\Gamma}}_n)_{j\cdot}) \ge \mu^\star - 2(d-2)\delta_n.
\end{equation*}
Assume now that $\mu^\star>2(d-2)\delta_n$. Since $t \mapsto \exp(-\zeta t^2)$ is bijective on $\R_+$, this implies
\begin{equation}
\label{eq:w-bound-2}
    |\hat w_{ij}^{(n)}| \le \exp\left[-\zeta \{\mu^\star - 2(d-2)\delta_n\}^2\right].
\end{equation}
Consider now the couple of indices $(i^*,j^*) \in V$ such that
\begin{equation*}
    \max_{1 \le i,j \le d} |\hat w^{(n)}_{ij} - w^\star_{ij}| = |\hat w_{i^*j^*}^{(n)} - w^\star_{i^*j^*}|,
\end{equation*}
Hence, on the event $\{\|\hat{\Gamma}_n - \Gamma^\star\|_\infty \le \delta_n\}$, if $\mu^\star > 2(d-2)\delta_n$, it follows from \eqref{eq:w-bound-1} or \eqref{eq:w-bound-2}, depending on whether $i^*$ and $j^*$ belong to the same cluster or not, that
\begin{equation*}
    \max_{1 \le i,j \le d} |\hat w^{(n)}_{ij} - w^\star_{ij}| \le a_n,
\end{equation*}
which proves the first part of Proposition \ref{prop:w-concentration}. As for the second part, observe that $\delta_n \to 0$ as $n\to \infty$, hence $\mu^\star > 2(d-2)\delta_n$ holds for sufficiently large $n$. Combining this with the conditions of Proposition \ref{prop:conc-vario},  there exists a sequence $(\varepsilon_n)_{n \in \N}$ converging to $0$, such that for $n$ large enough and any $\varepsilon \ge \varepsilon_n$ 
\begin{align*}
    \p\left(\max_{1 \le i,j \le d} |\hat w^{(n)}_{ij} - w^\star_{ij}| \le a_n\right) = & \, \p(||\hat \Gamma_n - \Gamma^\star||_\infty \le \delta_n) \\ 
    + & \, \p\left(\max_{1 \le i,j \le d} |\hat w^{(n)}_{ij} - w^\star_{ij}| \le a_n \, \Big| \, ||\hat \Gamma_n - \Gamma^\star||_\infty > \delta_n \right)\p(||\hat \Gamma_n - \Gamma^\star||_\infty > \delta_n) \\
    \ge & \, \p(||\hat \Gamma_n - \Gamma^\star||_\infty \le \delta_n) \\
    \ge & \,  1 - \varepsilon \, ,
\end{align*}
which leads to the desired convergence rate.
\begin{flushright}
    $\square$
\end{flushright}

\subsection{Proof of Theorem \ref{theo:estimator_consistency} and \ref{theo:partition_consistency}}

The proof of Theorem \ref{theo:estimator_consistency} is inspired by the beginning of the proof of Theorem 1 in \cite{kato2009asymptotics}. For two matrices $A = (a_{ij})$ and $B = (b_{ij})$ in $\mathcal{M}_{m,n}(\mathbb{R})$, we subsequently make use of the standard Hadamard product $A \odot B \in \mathcal{M}_{m,n}(\mathbb{R})$, which is defined 
element-wise by
\begin{equation*}
    (A \odot B)_{ij} = a_{ij} b_{ij}.
\end{equation*}. In particular, it relies on the well-known properties \begin{equation}
\label{eq:had1}
        (A+B) \odot C = (A\odot C) + (B \odot C),
    \end{equation}
    and
    \begin{equation}
    \label{eq:had2}
        ||A \odot B ||_\infty \le ||A||_\infty ||B||_\infty, 
    \end{equation}
    where $C \in \mathcal{M}_{m,n}(\mathbb{R})$.

Theorem \ref{theo:partition_consistency} then follows from Theorem \ref{theo:estimator_consistency}.

\begin{proof}[Proof of Theorem \ref{theo:estimator_consistency}]
Let $\mathcal{S}^1_d = \left\{\Theta \in \mathcal M_d(\R): \Theta^\top = \Theta, \, \Theta \vect{1}_d = \vect{0}_d\right\}$ be the subspace of symmetric matrices with zero row sums, and let $\mathcal S^{d-1} = \{\Theta \in \mathcal{S}^1_d \mid \|\Theta\|_\infty = 1 \}$ denote its unit sphere. For any $\varepsilon>0$, consider
\begin{equation*}
    \mathcal H_\varepsilon = \{\Theta \in \mathcal S^{d-1} \,:\, \Theta^\star + \varepsilon \Theta \in \SPd \},
\end{equation*}
which is non-empty since any $\Theta \in \SPd \cap \mathcal{S}^1_d$ belongs to this set.
Furthermore, let $\lambda > 0$. We define the quantity
\begin{equation*}
    \mu_\varepsilon = \inf_{\Theta \in \mathcal H_\varepsilon} \left\{ L_\pen^\star(\Theta^\star +\varepsilon \Theta, \lambda) - L_\pen^\star(\Theta^\star, \lambda) \right\}>0\,,
\end{equation*}
where $L_\pen^\star$ corresponds to the mapping $L_\pen$ in Equation \eqref{prob-nllh}  with $\hat \Gamma$ replaced by $\Gamma^\star$.
Let $\Theta \in \mathcal H_\varepsilon$. We consider $\delta_\Theta = \arg\min_{\delta \in (0,1]} \{\Theta^\star + \frac \varepsilon \delta \Theta \in \SPd\}$. Then let $\delta \in [
\delta_\Theta, 1]$ and set $\Delta_n(\Theta, \lambda) = L_\mathcal P^{(n)}(\Theta, \lambda) - L_\pen^\star(\Theta, \lambda)$. Since $L_\mathcal P^{(n)}$ is convex, the following inequality holds:
\begin{align*}
 \delta \left[ L_\mathcal P^{(n)}(\Theta^\star + \frac{\varepsilon}{\delta} \Theta, \lambda) - L_\mathcal P^{(n)}(\Theta^\star, \lambda)\right]  \ge & L_\mathcal P^{(n)}(\Theta^\star + \varepsilon \Theta, \lambda) - L_\mathcal P^{(n)}(\Theta^\star, \lambda)\\
   \ge &\left[ L_\pen^\star(\Theta^\star+ \varepsilon\Theta, \lambda) - L_\pen^\star(\Theta^\star, \lambda) \right] +\\ & \left[\Delta_n(\Theta^\star + \varepsilon\Theta, \lambda) - \Delta_n(\Theta^\star, \lambda) \right].
\end{align*}
Hence,  
\begin{equation}
    \label{eq:cvx-ineq}
    \delta \left[ L_\mathcal P^{(n)}(\Theta^\star + \frac{\varepsilon}{\delta} \Theta, \lambda) - L_\mathcal P^{(n)}(\Theta^\star, \lambda)\right] \ge \mu_\varepsilon - 2 \Delta_n,
\end{equation}
where $\Delta_n = \sup_{\{\Theta \in \SPd: || \Theta - \Theta^\star||_\infty\leq\varepsilon \}} |\Delta_n(\Theta, \lambda)|$. Suppose now that $\hat \Theta_n(\lambda) \in \SPd$ satisfies  $|| \hat \Theta_n(\lambda) - \Theta^\star||_\infty = \varepsilon/\delta$, for $\delta \in (0,1)$. There exists $\Theta \in \mathcal H_\varepsilon$ such that $\Theta_n(\lambda) = \Theta^\star + \frac{\varepsilon}{\delta}\Theta$. Since $L_\mathcal P^{(n)}(\cdot, \lambda)$ achieves its minimum at $\hat \Theta_n(\lambda)$ on $\SPd$, it follows that the left-hand term of the inequality \eqref{eq:cvx-ineq} is non-positive, and consequently, so is the right-hand term. By contraposition
\begin{equation*}
    \Delta_n < \frac{\mu_\varepsilon}{2} \Longrightarrow ||\hat\Theta_n(\lambda) - \Theta^\star||_\infty \le \varepsilon.
\end{equation*}
In addition, since the function $L_\pen^\star$ is $C^2$ for matrices of constant rank, by a second-order Taylor expansion with integral remainder we have for all $\Theta \in \mathcal H_\varepsilon$
\begin{equation*}
    L_\pen^\star(\Theta^\star +\varepsilon \Theta, \lambda) - L_\pen^\star(\Theta^\star, \lambda) = \varepsilon^2 \int_0^1  (1-t)\nabla^2L_\pen^\star(\Theta^\star+t\varepsilon \Theta \lambda)(\Theta,\Theta) dt,
\end{equation*}
where the first-order term vanishes because $\nabla L_\pen^\star(\Theta^\star, \lambda)(\Theta) = 0$.
Under Assumption 1, which directly yields that $L_\pen^\star$ is strongly convex, the following inequality holds 
\begin{equation*}
    L_\pen^\star(\Theta^\star +\varepsilon \Theta, \lambda) - L_\pen^\star(\Theta^\star, \lambda) \ge \varepsilon^2 S\int_0^1 1-t dt = \frac{S}{2}\varepsilon^2,
\end{equation*}
which induces 
\begin{equation*}
    \mu_\varepsilon \ge \frac{S}{2}\varepsilon^2,
\end{equation*}
and therefore
\begin{equation*}
    \Delta_n \le \frac{S}{4}\varepsilon^2 \Longrightarrow     ||\hat\Theta_n(\lambda) - \Theta^\star||_\infty \le \varepsilon.
\end{equation*}

We shall now derive a bound for $\Delta_n$. Let $\Theta \in \SPd$ be a matrix such that $||\Theta - \Theta^\star||_\infty \le \varepsilon$. Observing that the penalty term $\pen(\Theta)$ for both $L_\pen^{(n)}$ and $L_\pen^\star$ takes the form
$$
    \pen (\Theta) = \frac 12 \vect{1}_d^\top \left(W \odot D^2(\Theta)\right) \vect{1}_d,
    $$
and using the distributivity of the Hadamard product recalled in \eqref{eq:had1}, we have
\begin{equation}
    \label{eq:delta-bound-1}
|\Delta_n(\Theta, \lambda)| \le \frac{1}{2} \left| \text{tr} \big( (\hat\Gamma_n - \Gamma^\star)\Theta \big) \right| + \frac{\lambda}{2} \left| \vect{1}_d^\top \big( (\hat W_n - W^\star) \odot D^2(\Theta) \big) \vect{1}_d \right|.
\end{equation}
By Hölder's inequality, the trace term in \eqref{eq:delta-bound-1} can be bounded as follow
\begin{equation}
    \label{eq:trace-bound}
     |\text{tr}((\hat\Gamma_n-\Gamma^\star)\Theta)| \le ||\Theta||_1 ||\hat\Gamma_n-\Gamma^\star||_\infty
\end{equation}
Regarding the penalty term, applying the Cauchy-Schwarz inequality together with standard matrix norm inequalities yields
$$\begin{array}{lll}
     |\vect{1}_d^\top \left((\hat W_n - W^\star) \odot D^2(\Theta)\right) \vect{1}_d| & \le & ||\vect{1}_d||_2 \times ||((\hat W_n - W^\star) \odot D^2(\Theta)) \vect{1}_d||_2 \\ 
    & \le & d \,|||(\hat W_n - W^\star) \odot D^2(\Theta)|||\\
     & \le & d^2\, ||(\hat W_n - W^\star) \odot D^2(\Theta)||_\infty,
\end{array}$$
where, as a reminder, $\vert{}\vert{}\vert{}\cdot\vert{}\vert{}\vert{}$ denotes the spectral norm of a matrix.
Using the inequality \eqref{eq:had2}, we deduce 
\begin{equation}
\label{eq:pen-bound}
    |\vect{1}_d^\top \left((\hat W_n - W^\star) \odot D^2(\Theta)\right) \vect{1}_d| \le d^2 ||\hat W_n - W^\star||_\infty ||D^2(\Theta)||_\infty 
\end{equation}
Combining \eqref{eq:trace-bound} and \eqref{eq:pen-bound} with \eqref{eq:delta-bound-1} and applying Proposition 4, it follows that on the event $\{||\hat\Gamma_n-\Gamma^\star||_\infty \le \delta_n\}$
\begin{equation*}
    |\Delta_n(\Theta, \lambda)| \le \frac{1}{2} ||\Theta||_1 \delta_n + \frac \lambda 2 d^2 ||D^2(\Theta)||_\infty a_n.
\end{equation*}
In addition, we have
\begin{equation*}
    ||\Theta||_1 \le d^2 ||\Theta^\star - \Theta ||_\infty + ||\Theta^\star||_1 \le d^2 \varepsilon + ||\Theta^\star||_1,
\end{equation*}
and, for any $(i,j) \in V^2$, using the triangle inequality applied to the pseudometric $D$, \begin{equation}
    \label{eq:dist-inequation}
     D(\vect{\Theta}_{i\cdot},  \vect{\Theta}_{j\cdot}) \le  D(\vect{\Theta}_{i\cdot}, \vect{\Theta}^\star_{i\cdot}) +  D( \vect{\Theta}_{j\cdot}, \vect{\Theta}^\star_{j\cdot}) + D(\vect{\Theta}^\star_{i\cdot}, \vect{\Theta}^\star_{j\cdot}) \le 2(d-2) \varepsilon + \rho^\star,
\end{equation}
We thus obtain that
\begin{equation}
    \label{eq:pen-bound-2}
        \Delta_n \le \frac 12 (d^2 \varepsilon + ||\Theta^\star||_1) \delta_n + \frac \lambda 2 d^2  (2(d-2) \varepsilon + \rho^\star)^2 a_n.
\end{equation}

\paragraph{}
Finding a value for $\varepsilon$ such that the right-hand side of \eqref{eq:pen-bound-2} is strictly less than $\frac{S}{4} \varepsilon^2$ yields a concentration bound for $\hat{\Theta}_n(\lambda)$. Therefore, the goal is to find an $\varepsilon>0$ such that:

\begin{equation*}
    (d^2 \varepsilon + ||\Theta^\star||_1) \delta_n +  \lambda  d^2  (2(d-2) \varepsilon + \rho^\star)^2 a_n \le \frac{S}{2} \varepsilon^2.
\end{equation*}

\paragraph{}
The previous equation reduces to a quadratic inequality of the form $\alpha\varepsilon^2 + \beta\varepsilon + \gamma \le 0$ where \begin{equation*}
    \alpha = 4\lambda d^2(d-2)^2a_n -\frac{S}{2}, \quad \beta = d^2 \delta_n + \lambda d^2  4(d-2) \rho^\star a_n, \quad \gamma = \lambda (d\rho^\star)^2 a_n +||\Theta^\star||_1 \delta_n.
\end{equation*}
Suppose that $\lambda \le \frac{S}{8d^2(d-2)^2} a_n^{-1}$. Hence, $\alpha \le 0$ and the previous inequality holds for all values outside the roots of the polynomial. In particular, it is satisfied for any $\varepsilon \geq \varepsilon_1 = \frac{-\beta - \sqrt{\Lambda}}{2\alpha} \ge 0$, where $\Lambda = \beta^2 - 4\alpha \gamma>0$. Since \begin{equation*}
    b_n(\lambda) = \frac{\delta_n + \lambda 4(d-2) \rho^\star a_n + \sqrt{\Lambda'_n}}{\frac{S}{d^2} - 8\lambda (d-2)^2a_n} \ge \varepsilon_1,
\end{equation*} 
this concludes the first part of Theorem \ref{theo:estimator_consistency}. Regarding the second part of the proof, recall that $(a_n)_n$ vanishes as $n \to \infty$, which guarantees that $\lambda \le \frac{S}{8d^2(d-2)^2} a_n^{-1}$ for sufficiently large $n$. As $n \to \infty$, the denominator of $b_n(\lambda)$ tends to $S/d^2$, whereas the numerator behaves asymptotically as $\left(\frac{2S}{d^2}\left( \frac{||\Theta^\star||_1}{d^2} \delta_n + \lambda (\rho^\star)^2 a_n \right)\right)^{1/2}$. Consequently,
\begin{equation*}
b_n(\lambda) = O\left(\sqrt{\delta_n + \lambda \tilde{C} a_n}\right),
\end{equation*}
for any $\lambda>0$, where $\tilde{C} = \frac{d^2 (\rho^\star)^2}{||\Theta^\star||_1}$. Under the conditions of Proposition \ref{prop:conc-vario}, proceeding as in the proof of Proposition \ref{prop:w-concentration} thus gives the second part of Theorem \ref{theo:estimator_consistency}.

\end{proof}

\begin{proof}[Proof of Theorem \ref{theo:partition_consistency}]

For any pair $(i,j) \in T^\star$, since $D(\vect{\Theta}^\star_{i\cdot}, \vect{\Theta}^\star_{j \cdot}) = 0$, it follows from the first inequality in Equation~\eqref{eq:dist-inequation} that any $\Theta \in \SPd$ satisfies
\begin{equation*}
D(\vect{\Theta}{i\cdot}, \vect{\Theta}{j \cdot}) \le 2(d-2)|\Theta - \Theta^\star|_\infty.
\end{equation*} Reversing the roles of $\Theta$ and $\Theta^\star$ in \eqref{eq:dist-inequation} likewise leads to
\begin{equation*}
    D(\vect{\Theta}^\star_{i\cdot}, \vect{\Theta}^\star_{j \cdot}) \le 2(d-2)||\Theta - \Theta^\star||_\infty + D(\vect{\Theta}_{i\cdot}, \vect{\Theta}_{j \cdot})
\end{equation*}
for $(i,j) \notin T^\star$. Applying the first part of Theorem \ref{theo:estimator_consistency} directly yields the first part of Theorem \ref{theo:partition_consistency}, from which the second part readily follows.
\end{proof}

\section{Additional simulation results}

\subsection{Approximate block structure}
\label{appendix:approximate.sim}

We investigate the robustness of the HRC procedure to departures from the exact block-model assumption. Starting from the balanced configuration detailed in Example \ref{example_block_structure_Theta}, independent uniform perturbations are added to the non-zero coefficients of the matrix $\Theta^\star$. Although the resulting matrix no longer satisfies an exact block structure, it remains close to the original model.

The noisy scenario, shown in Figure \ref{fig:noised-res15}, investigates the robustness of the proposed methodology when the exact block assumption is violated. Although the perturbation substantially alters the HR precision matrix, the latent block organisation remains sufficiently pronounced for the HRC procedure to recover it accurately. This suggests that our algorithm does not merely fit an idealised block model but also tolerates moderate departures from it.

\begin{figure}[!ht]
    \centering
    \includegraphics[width=8cm]{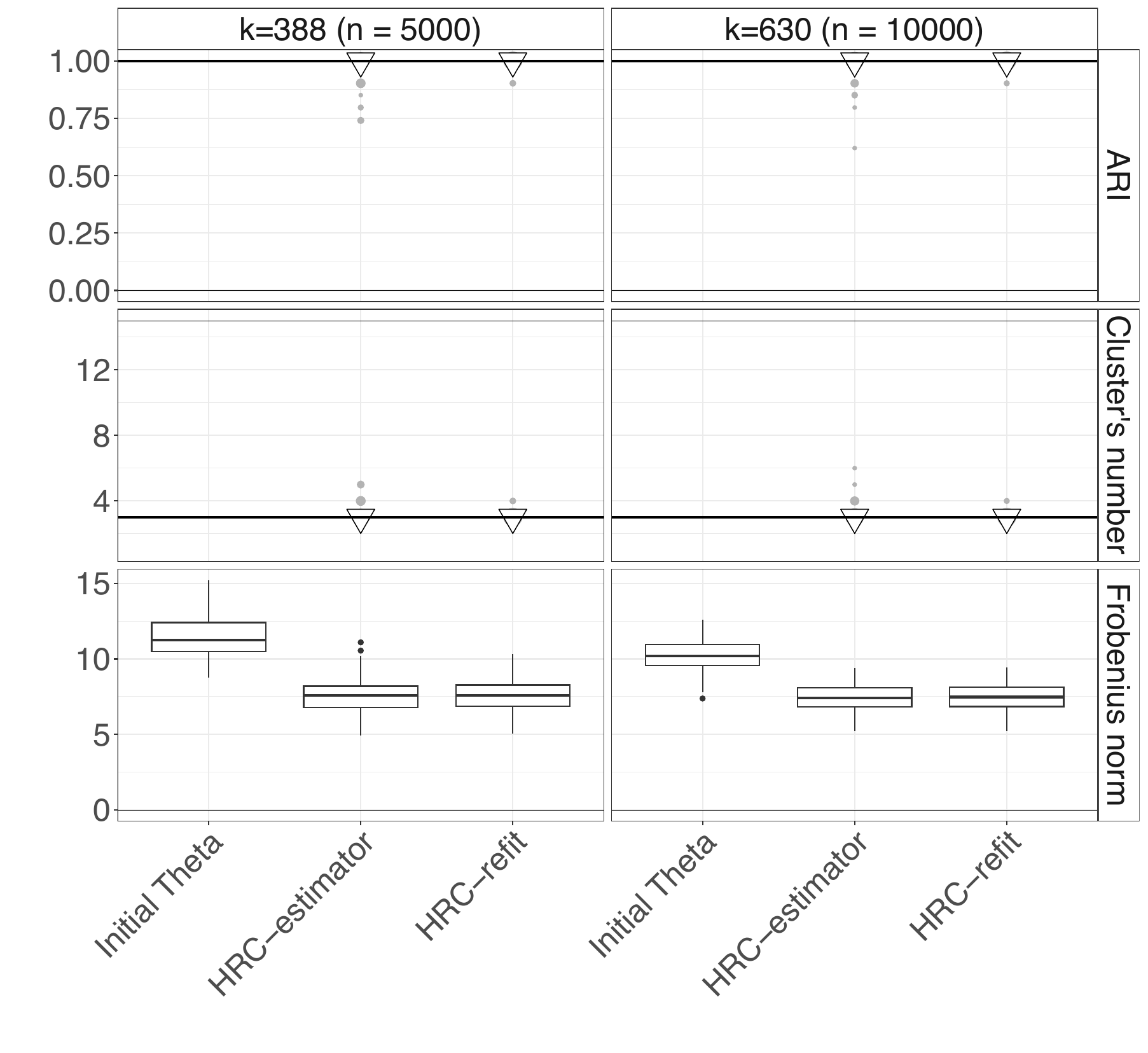}
    \caption{Clustering performance (top two rows) and HR precision matrix estimation error (bottom row) for a balanced noisy structure for $d=15$ for different sample sizes. The Frobenius norm is calculated with respect to the true block-structured HR precision matrix. See Figure 1 for further explanation on the graphs.}
    \label{fig:noised-res15}
\end{figure}

\subsection{Summary table of results}
\label{appendix:table.results}
To facilitate the comparison across all configurations, Table \ref{tab:results} summarises the average performances of the refitted estimator for the two sample sizes, $n=5000$ and $n=10000$, respectively. As already discussed in Section \ref{sect-sim}, for $d=15$, the partition is accurately recovered, even in a noised scenario and a moderate sample size ($n=5000$). When $d=60$, the partition estimation remain accurate for the balanced structure. This becomes more challenging in the unbalanced case, especially when the sample size is small. Increasing the latter address this issue.

\begin{table*}[!ht]
    \tabcolsep=0pt
    \begin{tabular*}{\textwidth}{@{\extracolsep{\fill}}lcccc@{\extracolsep{\fill}}}
        \toprule%
        $d= 15$ &\multicolumn{2}{@{}c@{}}{$n=5000$} & \multicolumn{2}{@{}c@{}}{$n=10000$} \\
        \cline{2-3}\cline{4-5}%
         Configuration & ARI & $\hat{K}$ & ARI & $\hat K$ \\
        \midrule
        $d_1=d_2=d_3=5$ & 0.997 (0.017) & 3.03 (0.171) & 0.999 (0.009) & 3.01 (0.1)\\
        $d_1= 5$, $d_2=3$, $d_3=7$  & 0.995 (0.025) & 3.06  (0.278) & 0.998 (0.016) & 3.02  (0.141) \\
        $d_1=d_2=d_3=5$ (noised)  & 0.997 (0.017) & 3.03  (0.171) & 0.998 (0.014) & 3.02 (0.141)\\
        \midrule
        $d=60$ & & & & \\
        \midrule
        $d_1=\dots=d_5=12$  &  0.999 (0.007) & 5.01 (0.1) & 1 (0)  &  5 (0)\\
        $d_1= 10$, $d_2=20$, $d_3=30$ &0.629 (0.385) &  2.39 (0.803) &   0.964 (0.107) & 2.98  (0.2)\\
        \botrule
    \end{tabular*}
    \caption{The quantities $d_i$ refer to the size of the clusters. For each configuration we provide the mean value of the number of clusters and the ARI for the Refit HRC estimator. The empirical standard deviation is specified in brackets.}
    \label{tab:results}
\end{table*}

\section{Additional details on Application}

\subsection{Table of currency codes}
\label{appendix:table.currency}

Table \ref{tab:country} shows the country codes used in Figure \ref{fig:currency} and the associated country name.

\begin{table}[!ht]
    \centering
    \begin{tabular}{c|l|c|l}
        \textbf{Code} & Country & \textbf{Code} & Country\\
        \hline
        AUS & Australia Dollar & NOR & Norwegian Krone\\
        CAN & Canadian Dollar & POL & Polish Zloty \\
        CHN & Chinese Yuan & RUS &  Russian Ruble \\
        CZE & Czech Koruna & SAU & Saudi Riyal \\
        DNK & Danish Krone & SPG & Singapore Dollar \\
        EUR & Euro & ZAF & South African Rand \\
        HKG & Hong Kong Dollar & KOR & South Korean Won \\
        HUN & Hungarian & SWE & Swedish Krona \\
        IND & Indian Rupee & CHE & Swiss Franc \\
        ISR & Israeli Shekel & TWN & Taiwan Dollar \\
        JPN & Japanese Yen & THA & Thai Baht \\
        MYS & Malaysian Ringgit & TUR & Turkish Lira \\
        NZL & New Zealand Dollar & USA & US Dollar \\
    \end{tabular}
    \caption{Country codes of foreign exchanges rate (into GBP).}
    \label{tab:country}
\end{table}

\subsection{Results on dependency on the clusters}
\label{appendix:table.application}

Table \ref{tab:app-chi} presents the estimation values for the joint tail dependence coefficient obtained with the clustering induced by the groups shown in Figure \ref{fig:currency}. The table shows groups with an inner joint tail dependence equal to one: these are exactly the groups containing a single variable.

\begin{table}[!ht]
    \centering
    \begin{tabular}{|l|r|r|r|r|r|r|r|r|}
        \hline
         Group & 1 & 2 & 3 & 4 & 5 & 6 & 7 & 8\\
        \hline
        1 & \textbf{0.977} & 0.429 & 0.399 & 0.438 & 0.437 & 0.444 & 0.621 & 0.541\\
        \hline
        2 & 0.429 & \textbf{0.593} & 0.456 & 0.415 & 0.412 & 0.429 & 0.418 & 0.451\\
        \hline
        3 & 0.399 & 0.456 & \textbf{0.497} & 0.385 & 0.402 & 0.422 & 0.384 & 0.433\\
        \hline
        4 & 0.438 & 0.415 & 0.385 & \textbf{1.000} & 0.475 & 0.421 & 0.491 & 0.394\\
        \hline
        5 & 0.437 & 0.412 & 0.402 & 0.475 & \textbf{0.826} & 0.527 & 0.450 & 0.399\\
        \hline
        6 & 0.444 & 0.429 & 0.422 & 0.421 & 0.527 & \textbf{0.498} & 0.435 & 0.408\\
        \hline
        7 & 0.621 & 0.418 & 0.384 & 0.491 & 0.450 & 0.435 & \textbf{1.000} & 0.470\\
        \hline
        8 & 0.541 & 0.451 & 0.433 & 0.394 & 0.399 & 0.408 & 0.470 & \textbf{0.525}\\
        \hline
    \end{tabular}
    \caption{Estimation values of the joint tail dependence coefficient $\chi$ within (in bold) and between each group of currencies.}
    \label{tab:app-chi}
\end{table}

%% file: figures/block-graph.tex
\begin{tikzpicture}[auto, thick]


    \node[draw, circle, fill=black, inner sep=2pt, ] (a11) at (0,5) {};
    \node[draw, circle, fill=black, inner sep=2pt, ] (a12) at (1.5,5) {};
    \node[draw, circle, fill=black, inner sep=2pt, ] (a13) at (0,3.5) {};
    \node[draw, circle, fill=black, inner sep=2pt, ] (a14) at (1.5,3.5) {};

    \foreach \i/\j in {a11/a12,a11/a13,a11/a14,a12/a13,a12/a14,a13/a14}
      \draw (\i) -- (\j);
    \draw (a13) -- (a14) node[midway, above] {$r_{11}$};
    \node[draw, circle, fill=black, inner sep=2pt] (a21) at (0,2) {};
    \node[draw, circle, fill=black, inner sep=2pt, ] (a22) at (1.5,2) {};
    \draw (a21) -- (a22) node[midway, below] {$r_{22}$};

    \draw[black, dashed] (a21.north) .. controls +(0,0) and +(-1,0) .. (a11);

    \draw[black, dashed] (a22) .. controls +(2,2) and +(2,2) .. (a11) node[midway, right] {$r_{12}$};
    \draw[black, dashed] (a12) .. controls +(-2,2) and +(-2,2) .. (a21);
    \draw[black, dashed] (a22.north) .. controls +(0,0) and +(1,0) .. (a12);
    \draw[black, dashed] (a21.north) .. controls +(0,0) and +(0,0) .. (a13.south);
    \draw[black, dashed] (a21.north) .. controls +(1,0) and +(0,0) .. (a14.south);
    \draw[black, dashed] (a22.north) .. controls +(-1,0) and +(0,0) .. (a13.south);
    \draw[black, dashed] (a22.north) .. controls +(0,0) and +(0,0) .. (a14.south);

    \node[draw, circle, fill=black, inner sep=2pt, ] (a31) at (0,-1) {};
    \node[draw, circle, fill=black, inner sep=2pt, ] (a32) at (1.5,-1) {};
    \node[draw, circle, fill=black, inner sep=2pt, ] (a33) at (0.75,0) {};

    \foreach \i/\j in {a31/a32,a31/a33,a32/a33}
      \draw (\i) -- (\j) node[midway, below] {};

    \draw (a31) -- (a32) node[midway, above] {$r_{33}$};

    \foreach \j in {a21,a22}
         \draw[black, dotted] (\j) .. controls +(0,-1) and +(0,0) .. (a33.north);

    \draw[black, dotted] (a21.south) .. controls +(0,0) and +(0,1) .. (a32) ;
    \draw[black, dotted] (a22.south) .. controls +(0,0) and +(0,1) .. (a31);
    \draw[black, dotted] (a31.north) .. controls +(0,0) and +(-1,0) .. (a21) node[midway, left] {$r_{23}$};
    \draw[black, dotted] (a32.north) .. controls +(0,0) and +(1,0) .. (a22);



    \node[draw, circle, minimum size=2cm] (C1) at (9,5) {};
    \node[fill=black, circle, inner sep=2pt] (c11) at ($(C1.center)+(-0.5,0.5)$) {};
    \node[fill=black, circle, inner sep=2pt] (c12) at ($(C1.center)+(0.5,0.5)$) {};
    \node[fill=black, circle, inner sep=2pt] (c13) at ($(C1.center)+(0.5,-0.5)$) {};
    \node[fill=black, circle, inner sep=2pt] (c14) at ($(C1.center)+(-0.5,-0.5)$) {};
    \foreach \i/\j in {c11/c12,c11/c13,c11/c14,c12/c13,c12/c14,c13/c14}
      \draw (\i) -- (\j);

    \node[draw, circle, minimum size=2cm] (C2) at (9,2) {};
    \node[fill=black, circle, inner sep=2pt] (c21) at ($(C2.center)+(-0.3,0)$) {};
    \node[fill=black, circle, inner sep=2pt] (c22) at ($(C2.center)+(0.3,0)$) {};
    \draw (c21) -- (c22);

    \node[draw, circle, minimum size=2cm] (C3) at (9,-1) {};
    \node[fill=black, circle, inner sep=2pt] (c31) at ($(C3.center)+(-0.4,-0.2)$) {};
    \node[fill=black, circle, inner sep=2pt] (c32) at ($(C3.center)+(0.4,-0.2)$) {};
    \node[fill=black, circle, inner sep=2pt] (c33) at ($(C3.center)+(0,0.4)$) {};
    \foreach \i/\j in {c31/c32,c31/c33,c32/c33}
      \draw (\i) -- (\j);

    \draw[black, dashed] (C1.south) -- (C2.north) node[midway,right]{$r_{12}$};
    \draw[black, dotted] (C2.south) -- (C3.north) node[midway,right]{$r_{23}$};

  \end{tikzpicture}